\documentclass{amsart}
\usepackage{amsmath}
\usepackage{amsthm}
\usepackage{amssymb}
\usepackage{graphics}
\usepackage{graphicx}
\usepackage{bm}

\usepackage{amsfonts}
\usepackage{latexsym}
\usepackage{mathrsfs}
\usepackage{eucal}
\usepackage{color}

\newcommand{\be}{\begin{equation}}
\newcommand{\ee}{\end{equation}}
\newcommand{\ba}{\begin{eqnarray}}
\newcommand{\ea}{\end{eqnarray}}
\newcommand{\bi}{\begin{itemize}}
\newcommand{\ei}{\end{itemize}}
\newcommand{\bn}{\begin{enumerate}}
\newcommand{\en}{\end{enumerate}}
\newcommand{\bbm}{\begin{bmatrix}}
\newcommand{\ebm}{\end{bmatrix}}
\newcommand{\bpm}{\begin{pmatrix}}
\newcommand{\epm}{\end{pmatrix}}
\newcommand{\bp}{\begin{proof}}
\newcommand{\ep}{\end{proof}}
\newcommand{\nn}{\nonumber}

\newcommand{\mr}{\ensuremath{\mathrm}}
\newcommand{\scr}{\ensuremath{\mathscr}}
\newcommand{\mbf}{\ensuremath{\mathbf}}
\newcommand{\mc}{\ensuremath{\mathcal}}

\newcommand{\ov}{\ensuremath{\overline}}

\newcommand{\Ga}{\ensuremath{\Gamma}}
\newcommand{\ga}{\ensuremath{\gamma}}
\newcommand{\Om}{\ensuremath{\Omega}}

\newcommand{\la}{\ensuremath{\lambda }}

\def\C{\mathbb{C}}

\def\D{\mathbb{D}}

\def\N{\mathbb{N}}
\def\B{\mathbb{B}}

\renewcommand{\H}{\ensuremath{\mathcal{H} }}
\newcommand{\J}{\ensuremath{\mathcal{J} }}

\newcommand{\K}{\ensuremath{\mathcal{K} }}
\renewcommand{\L}{\ensuremath{\mathcal{L} }}

\newcommand{\F}{\ensuremath{\mathbb{F} }}

\def\kz{K \{ Z , y , v \}}
\def\kw{K \{ W , x , u \}}
\newcommand{\zyv}[1]{\ensuremath{\left\{ {#1}  \right\} }}

\def\tr{\mathrm{tr}}

\newcommand\ssim{\mathrel{%
  \ooalign{\raise0.4ex\hbox{\scalebox{0.8}{$\subset$}}\cr\hidewidth\raise-0.6ex\hbox{\scalebox{0.9}{$\sim$}}\hidewidth\cr}}}

\newcommand{\ip}[2]{\ensuremath{\left\langle {#1} , {#2} \right\rangle}}
\newcommand{\ipcn}[2]{\ensuremath{\left( {#1} , {#2} \right) _{\C ^n}}}
\newcommand{\ipc}[2]{\ensuremath{\left( {#1} , {#2} \right) _{\C ^{n+1}}}}
\newcommand{\dom}[1]{\ensuremath{\mathrm{Dom} ({#1}) }}
\renewcommand{\dim}[1]{\ensuremath{\mathrm{dim} \left( {#1} \right) }}
\newcommand{\ran}[1]{\ensuremath{\mathrm{Ran} \left( {#1} \right) }}
\renewcommand{\ker}[1]{\ensuremath{\mathrm{Ker} \left( {#1} \right) }}

\newcommand{\re}[1]{\ensuremath{\mathrm{Re} \left( {#1} \right) }}

\numberwithin{equation}{section}
\numberwithin{subsection}{section}

\newtheorem{thm}[subsection]{Theorem}
\newtheorem*{thm*}{Theorem}

\newtheorem{lemma}[subsection]{Lemma}
\newtheorem{prop}[subsection]{Proposition}
\newtheorem{cor}[subsection]{Corollary}

\theoremstyle{definition}
\newtheorem{defn}[subsection]{Definition}
\newtheorem{remark}[subsection]{Remark}

\title[Operators affiliated to the free shift]{Operators affiliated to the free shift on the free Hardy space}

\author{Michael T. Jury}
\address{University of Florida}
\email{mjury@ufl.edu}

\author{Robert T.W. Martin}
\address{University of Cape Town}
\email{rtwmartin@gmail.com}

\begin{document}
\small
\date{\today}
\bibliographystyle{plain}
\maketitle
%\onehalfspace

\begin{abstract}
The Smirnov class for the classical Hardy space is the set of ratios of bounded analytic functions on the open complex unit disk with outer denominators. This definition extends naturally to the commutative and non-commutative multi-variable settings of the Drury-Arveson space and the full Fock space over $\C ^d$. Identifying the Fock space with the free multi-variable Hardy space of non-commutative or free holomorphic functions on the non-commutative open unit ball, we prove that any closed, densely-defined operator affiliated to the right free multiplier algebra of the full Fock space acts as right rmultiplication by a function in the right free Smirnov class (and analogously, replacing "right" with "left"). 
\end{abstract}

\section{Introduction}

As shown in \cite{Suarez}, any densely-defined closed linear operator, $T$, commuting with multiplication by $z$ on $H^2 (\D )$, the Hardy space of the unit disk, can be viewed as
$T = \ov{M _{b/a}}$, the closure of multiplication by $b/a$ where $a,b$ are bounded analytic functions in the open complex unit disk, $\D$, and $a$ is outer. That is, $b/a$ belongs to the \emph{Smirnov class} $\scr{N} ^+ $:
$$ \scr{N} ^+  = \left\{ \left. \frac{b}{a} \right| \ a,b \in H^\infty (\D ), \ \ov{\ran{M_a}} = H^2 (\D ) \right\} = \frac{ H^\infty (\D ) }{\scr{O} ^\infty (\D ) }. $$ Here, $\scr{O} ^\infty (\D)$ denotes the outer functions in $H^\infty (\D )$, the algebra of bounded analytic functions in $\D$. In particular, also as proven in \cite{Suarez}, $H^2 (\D) \subset \scr{N} ^+$: Given any $h \in H^2 (\D)$, one can find $b,a \in H^\infty (\D)$ with $a$ an outer function so that $1/a \in H^2 (\D)$ and 
$h = b/a $.  Recall that $a \in H^\infty (\D)$ is outer if multiplication by $a$, $M_a : H^2 (\D ) \rightarrow H^2 (\D )$ has dense range; equivalently $a \in H^2 (\D) \cap H^\infty (\D)$ is cyclic for the \emph{shift}, $S$, the operator of multiplication by $z$. This result was extended to the closed operators affiliated to any compressed shift $S_u$, where $u \in H^\infty (\D)$ is inner (compressed to the range of multiplication by $u$) by Sarason in \cite{Sarason-ub}. Also see \cite{Mar2009,Berc2009} for extensions to the class of $C_0$ contractions.

The Drury-Arveson space, $H^2 _d$, can be viewed as the several-variable Hardy space, it is the unique reproducing kernel Hilbert space (RKHS) associated to the several-variable Szeg\"{o} kernel on the multi-variable open unit ball $\B ^d := (\C ^d ) _1$. In \cite[Theorem 10.3]{Alpay}, it was shown that elements of $H^2 _d$ can also be identified with elements of the several-variable Smirnov class, $\scr{N} _d ^+$.  Namely, given any $h \in H^2 _d$,
there are Drury-Arveson multipliers $b, a$ so that
$$ h = \frac{a}{1-b}, $$ $b$ is a contractive multiplier, $b(0) =0$ and $1-b$ is outer \cite[Theorem 1.1]{Smirnov}. (Recall that if $b$ is contractive then $1-b$ is necessarily \emph{outer}, \emph{i.e.} $M_{(1-b)}$ has dense range in $H^2 _d$ \cite[Lemma 2.3]{Smirnov}).  In \cite{JM-H2Smirnov} we gave a different proof of the inclusion $H^2 _d\subset \scr{N} _d ^+$, which showed that every $h\in H^2_d$ can be expressed in the form $h=b/a$, where $a,b$ are bounded multipliers, $a$ is outer and, in addition, $1/a\in H^2_d$. Our proof was based on the proof of an analogous fact for free Hardy space functions, which is part of the motivation for the present paper.

In this paper we extend these results to the non-commutative, multi-variable setting of the full Fock space, $F^2 _d$ over $\C ^d$.  The full Fock space can be regarded as the canonical non-commutative multi-variable Hardy space. Indeed, one can identify $F^2 _d$ with $H^2 (\B ^d _\N )$, the \emph{free Hardy space} of free holomorphic functions on the non-commutative multi-variable open unit ball $\B ^d _\N := \bigsqcup _{n=1} ^\infty \left( \C ^{n\times n} \otimes \C ^d \right) _1$ \cite{KVV,BMV,Pop-freeholo}. In direct analogy to the classical setting, $H^2 (\B ^d _\N )$ can be defined as the unique non-commutative reproducing kernel Hilbert space (NC-RKHS) on the non-commutative (NC) domain $\B ^d _\N$ corresponding to a natural completely positive non-commutative (CPNC) kernel, the NC Szeg\"{o} kernel. We will briefly recall the basics of non-commutative function theory and NC-RKHS theory in the upcoming Section \ref{prelim}. 

  In \cite{JM-H2Smirnov} we showed that if $T$ is a closed, densely deifned operator in the free Hardy space, which acts by left multiplication by some free function $H$, then $H$ must be a so-called {\em free Smirnov} function (see Definition~\ref{def:free-smirnov} below). The main goal of the present paper is to obtain the same conclusion from a weaker hypothesis--we will assume only that $T$ is closed, densely defined, and commutes with the free shift, and conclude that $T$ acts as multiplication by $H$ for some Smirnov $H$ (there are analagous left and right versions).  Even in the one variable case it takes some work to show that these hypotheses imply that $T$ is given by a multiplication, see \cite{Suarez} .  The lack of commutativity means that the one-variable proof given in \cite{Suarez} does not carry over. Our proof will exploit the free Beurling theorem in the same way as \cite{JM-H2Smirnov}, the main difficulty under the present weaker hypotheses is in showing that the relevant wandering subspace is one-dimensional. For this we exploit some general results about free deBranges-Rovnyak type spaces and column extreme free multipliers developed in \cite{JMfreeNC}.

\section{Preliminaries} \label{prelim}
All Hilbert space inner products will be conjugate linear in their first argument. If $X$ is a Banach space, $(X)_1$ and $[X]_1$ denote the open and closed unit balls of $X$, respectively.

\subsection{The full Fock space}

Recall that the full Fock space over $\C ^d$, $F^2 _d$, is the direct sum of all tensor powers of $\C ^d$:
\ba F^2 _d & := & \C \oplus \left( \C ^d \otimes \C ^d \right) \oplus \left( \C ^d \otimes \C ^d \otimes \C ^d \right) \oplus \cdots \nn \\
& =& \bigoplus _{k=0} ^\infty \left( \C ^d \right) ^{k \cdot \otimes }. \nn \ea Fix an orthonormal basis $\{ e_1, ... , e_d \}$ of $\C ^d$. The left creation
operators $L_1, ..., L_d$ are the operators which act as tensoring on the left by these basis vectors:
$$ L_k f := e_k \otimes f; \quad \quad f \in F^2 _d, $$ and similarly the right creation operators $R_k; \ 1\leq k \leq d$ are defined by tensoring on the right $$ R_k f := f \otimes e_k. $$ The left and right free shifts are the row operators $L := (L_1 , ... , L_d)$ and $R := (R_1, ... , R _d )$ which map
$F^2 _d \otimes \C ^d$ into $F^2 _d$. Both $L, R$ are in fact row isometries: $L ^* L = I_{F^2} \otimes I_d = R^* R$. It follows that the component shifts are also isometries with pairwise orthogonal ranges.  The orthogonal complement of the range of $L$ or $R$ is the vacuum vector $1$ which spans the the subspace $\C =: (\C^d) ^{0\cdot \otimes} \subset F^2 _d$. A canonical orthonormal basis for $F^2 _d$ is then $\{ e_\alpha \} _{\alpha \in \F ^d}$ where $e_\alpha = L^\alpha 1 = R^\alpha 1$ and $\F ^d$ is the free unital semigroup or monoid on $d$ letters. Here, if $\alpha = i_1 \cdots i_n \in \F ^d$, we use the standard notation $L^\alpha = L_{i_1} L_{i_2} \cdots L_{i_n}$.

Recall here that the free monoid, $\F^d$, on $d \in \N$ letters, is the multiplicative semigroup of all finite products or \emph{words} in
the $d$ letters $\{1, ... , d \}$. That is, given words $\alpha := i_1 ... i_n$, $\beta := j_1 ... j_m$, $i_k, j_l \in \{1 , ... , d \}; \ 1\leq k \leq n, \ 1 \leq l \leq m$, their product $\alpha \beta $ is defined by concatenation:
$$ \alpha \beta = i_1 ... i_nj_1 ... j_m, $$ and the unit is the empty word, $\emptyset$, containing no letters. Given $\alpha = i_1 \cdots i_n$, we use the standard notation $|\alpha | = n$ for the length of the word $\alpha$. The transpose map $\dag : \F ^d \rightarrow \F ^d$, defined by 
$$  i_1 \cdots i_d =\alpha \mapsto \alpha ^\dag := i_d \cdots i_1,  \quad \quad \mbox{is an involution.} $$ 

Define $L^\infty _d := \mr{Alg}(I,L) ^{-WOT}, \ R^\infty _d := \mr{Alg}(I,L) ^{-WOT}$, the left (\emph{resp.} right) free analytic Toeplitz algebra.  The \emph{transpose unitary}, $U _{\dag} : F^2 _d \rightarrow F^2 _d$, defined by $e_{\alpha } \mapsto e_{\alpha ^{\dag}}$ is a unitary involution of $F^2 _d$, and it is easy to verify that
$$ U_\dag L_k U_\dag^* = R_k, $$ so that adjunction by $U_\dag$ implements a unitary isomorphism between $L^\infty _d$ and $R^\infty _d$.

\subsection{The free Hardy space}

It will be convenient to view $F^2 _d$ as a non-commutative reproducing kernel Hilbert space (NC-RKHS) \cite{BMV} of freely non-commutative holomorphic functions on the non-commutative open unit ball \cite{KVV} (we will use the same notations as in \cite{JMfreeNC}):
$$ \B _\N ^d := \bigsqcup _{n=1} ^\infty \B ^d_n ; \quad \quad \B ^d_n   := \left( \C ^{n \times n} \otimes \C ^d \right) _1. $$ Elements of $\B ^d_n $ are viewed as strict row contractions on $\C ^n$.
Recall that for any complex vector space $V$,
$$ V_{nc} := \bigsqcup V_n; \quad \quad V_n := V \otimes \C ^{n\times n} =: V^{n \times n}. $$
The NC unit ball $\B ^d _\N$ is an example of a NC set: A set $\Om \subseteq V _{nc}$ is an NC set if it is closed under direct sums, and one writes:
$$ \Om =: \bigsqcup \Om _n; \quad \quad \Om _n := \Om \cap V_n. $$ A function $f : \Om \rightarrow \C _{nc}$ is called a NC or \emph{free function} if: 
$$ f : \Om _n \rightarrow \C ^{n\times n}; \quad \quad f \ \mbox{respects the grading}, $$ and
if $X \in \C ^{n\times m }, Z \in X_n , W \in X _m$ obey $ZX = XW$, then, 
$$ f(Z) X = X f(W); \quad \quad f \ \mbox{respects intertwinings.}$$

As shown in \cite{BMV}, $F^2 _d =  H^2 (\B _\N ^d )$ can be viewed as the \emph{free Hardy space} of the multi-variable NC unit ball $\B _\N ^d$, \emph{i.e.}
$H^2 (\B _\N ^d) = \H _{nc} (K)$ is the unique NC-RKHS corresponding to the NC-Szeg\"{o} kernel:
$K : \B _\N ^d \times \B _\N ^d \rightarrow  \L (\C _{nc} )$ defined by:
$$ K (Z, W)  [P ] := \sum _{\alpha \in \F ^d} Z^\alpha P (W^* ) ^{\alpha ^\dag}; \quad \quad Z \in \B ^d_n  , W \in \B ^d_m , P \in \C ^{n\times m}.$$
See \cite{BMV} for the full definition and theory of NC kernels. In particular, any NC kernel respects the grading and intertwinings in both arguments \cite[Section 2.3]{BMV}.

One can show that elements of $H^2 (\B _\N ^d ) := \H _{nc} (K)$ are locally bounded (and hence automatically) holomorphic free functions on $\B ^d _\N$ \cite[Chapter 7]{KVV}.  That is, any $f\in H^2 (\B _\N ^d )$ is Fr\'{e}chet and  G\^{a}teaux differentiable at any point $Z \in \B ^d _\N$ and $f$ has a convergent power series expansion (Taylor-Taylor series) about any point. 

(Generally any) $\H _{nc} (K)$ is formally defined as the Hilbert space completion of the linear span:
$$ \bigvee _{Z \in \B ^d _n, \ y,v \in \C ^n } \kz, $$ where the $\kw$ are the free functions on $\B ^d _\N$, $ \kw : \B ^d _n \rightarrow \C ^{n\times n}$, defined by:
$$ \kw (Z) y := K(Z,W) [yu^*] x; \quad \quad W \in \B ^d _m, \ Z \in \B ^d _n; \ u,x \in \C ^m, \ y \in \C ^n.$$ Completion is with respect to the inner product:
\begin{align*} & \ip{\kz}{\kw}  :=   \ipcn{y}{K(Z,W)[vu^*]x}; \\
& Z \in \B ^d_n , \ v , y \in \C ^n; \ W \in \B ^d _m, \ u,x \in \C ^m. \end{align*}

These point evaluation vectors have a familiar reproducing property: $K(Z,y,v)$ is the unique vector in $H^2 (\B _\N ^d )$ such that for any $f \in H^2 (\B _\N ^d )$,
\begin{equation}\label{eqn:kz-reproducing-formula}
\ip{\kz }{f}  = \ipcn{y}{f(Z)v }. 
\end{equation}

For any $Z \in \B ^d_n$ one can also define a natural \emph{kernel map} $K_Z \in \L (\C ^{n \times n} , H^2 (\B ^d _\N) )$ as follows: Any $A \in \C ^{n \times n}$
can be written as a linear combination of the rank one outer products
$$ y   v^*  = \bbm y_1\\ \vdots \\ y_n  \ebm \bbm \overline{v_1}, \ & \cdots \ & , \overline{v_n} \ebm ; \quad \quad y \in \C ^n, v^* \in (\C^n)^*. $$
Then we define $K_Z$ on rank one matrices $yv^*$ by the formula
\begin{equation}\label{eqn:K-bigZ-def}
 K_Z ( yv ^*) := \kz \in H^2 (\B _\N ^d ). 
\end{equation}
 Let us check that $K_Z$ is well defined: the vectors $y$ and $v$ determining a rank one matrix $yv^*$ are unique up to the scaling $y\to \lambda y, v\to \overline{\lambda}^{-1}v$ where $\lambda$ is any nonzero complex number. From the reproducing formula (\ref{eqn:kz-reproducing-formula}), it is evident that the vector $K\{Z,y,v\}$ is invariant under such a scaling, and so the formula (\ref{eqn:K-bigZ-def}) is unambiguous.  If we view $\C ^{n \times n}$ as a Hilbert space equipped with the normalized trace inner product, then $K_Z : \C ^{n \times n } \rightarrow H^2 (\B ^d _\N )$ extends to a bounded linear map, and its Hilbert space adjoint is the point evaluation map at $Z$:
$$ K_Z ^* F = F(Z) \in \C ^{n \times n}. $$ 

The free Hardy space and the full Fock space are canonically isomorphic: Define $\scr{U} : F^2 _d \rightarrow  H^2 (\B ^d _\N ) $ by:
\begin{align*} x & :=  \sum _{\alpha \in \F ^d } x_\alpha L^\alpha 1 \stackrel{\scr{U}}{\mapsto} f_x \in H^2 (\B _\N ^d ), \\
 f_x (Z) &:=  \sum _{\alpha \in \F ^d} Z^\alpha x_\alpha; \quad \quad Z \in \B _\N ^d. \end{align*}
The inverse, $\scr{U} ^{-1}$, acts on kernel vectors as:
 \be \kz \stackrel{\scr{U} ^{-1} }{\mapsto} x[Z,y,v] := \sum _{\alpha \in \F ^d} \ip{Z^\alpha v}{y} L^\alpha 1 \in F ^2 _d. \label{pevalimg}\ee

\subsection{Left and Right free multipliers}
\label{leftrightmult}
As in the classical setting, given a NC-RKHS $\H _{nc} (K )$ on an NC set $\Om$ (\emph{e.g.} $\B ^d _\N$), it is natural to consider the left and right \emph{multiplier algebras} $$ \mr{Mult} ^L (\H _{nc} (K) ), \ \mr{Mult} ^R (\H _{nc} (K) )$$ of NC functions on $\Om$ which left or (\emph{resp.}) right multiply $\H _{nc} (K)$ into itself. 
Namely, a free function $F$ on $\B _\N ^d$ is said to be a \emph{left free multiplier} if, for any $f \in H^2 (\B _\N ^d )$, $Ff \in H^2 (\B ^d _\N)$. Similarly, $G$ is called a \emph{right free multiplier}
if $fG \in H^2 (\B _\N ^d )$ for all $f \in H^2 (\B _\N ^d )$. As in the classical setting, the left and right \emph{free multiplier algebras}, $H^\infty _L (\B _\N ^d ) := \mr{Mult} ^L (H^2 (\B ^d _\N ) ), H^\infty _R (\B _\N ^d )$ are weak operator toplogy (WOT)-closed unital operator algebras. Moreover, adjunction by the canonical unitary $\scr{U}$ defines a unitary $*-$isomorphism of the left and right free analytic Toeplitz algebras $L^\infty _d , R^\infty _d$ onto these
left and right free multiplier algebras of $H^2 (\B _\N ^d )$.  As in the classical setting of $H^\infty (\D)$, the multiplier norm of any $F \in H^\infty _L (\B _\N ^d )$ can be computed as the supremum norm on the NC unit ball:
$$ \| F \| := \sup _{Z \in \B _\N ^d} \| F (Z ) \|. $$ The left and right Schur classes, $\scr{L} _d , \scr{R} _d$ are then defined as the closed unit balls of these left and right multiplier algebras (equivalently as the closed unit balls of $L^\infty _d , R^\infty _d$). 

Observe that if $F$ is a left free multiplier then,
\ba \ip{\kz}{Ff} & = & \ip{y}{F(Z) f(Z) v} \nn \\
& = & \ip{K \zyv{Z,F(Z) ^* y, v}}{f}, \nn \ea so that
\be (M^L _F)^* \kz = K \zyv{Z,F(Z) ^* y,v}, \label{leftpevalue} \ee and similarly, if $G$ is a right free multiplier,
\be (M^R _G)^* \kz = K \zyv{Z ,y,G(Z) v}. \label{rightpevalue} \ee Alternatively, using the kernel maps $K_Z$, we can write:
$$ (M_F ^L) ^* K_Z (yv) = K_W (F(Z)^*yv^*), \quad \quad \mbox{and} \quad \quad (M_G^R)^* K_Z = K_Z  (y v^*G(Z)^*). $$
One can check that if, \emph{e.g.}, right multiplication by $G(Z)$ is a right free multiplier then
$$(   (M_G^R)^* K_Z)^* ((M_G^R)^* K_W ) = K(Z,W)[G(Z) \cdot G(W) ^*]. $$

In particular, free holomorphic $F(Z),G (Z)$ belong to the left or right Schur classes if and only if
$$ K^F (Z,W)[\cdot]:= K(Z,W) - F(Z) K(Z,W)[\cdot] F(W) ^*  $$ or
$$ K^G (Z,W) [\cdot]:= K(Z,W) - K(Z,W)[G(Z) [\cdot] G(W) ^* ]  $$ are CPNC kernels, respectively. These NC kernels are called the left or right free deBranges-Rovnyak kernels of $F, G$ (\emph{resp.}) and in this case the corresponding NC-RKHS $\mc{H} _{nc} (K^F) =: \scr{H} ^L (F)$, $\mc{H} _{nc} (K^G) =: \scr{H} ^R (G)$ are the left and right free deBranges-Rovnyak spaces of $F,G$.

\subsection{Left vs. Right}

Any element $F \in L^\infty _d$ can be identified with the left free Fourier series:
$$ F \sim F(L) := \sum _{\alpha \in \F ^d}  F_\alpha L^\alpha ; \quad \quad F_\alpha := \ip{L^\alpha 1}{F 1}. $$
That is, $F$ is identified with its \emph{symbol}:
$$ f:= F1 = \sum _{\alpha \in \F ^d}  F_\alpha L^\alpha 1 \in F^2 _d, $$
and we say that $F(L) = M^L _{f}$ acts as left multiplication by $f=F1$. In general the free Fourier series does not converge in SOT or WOT, but the Ces\`aro
sums converge in the strong operator toplogy (SOT) to $F$ \cite{DPac}.

Similarly, in the operator valued setting, any $F \in L^\infty _d \otimes \L (\H , \J)$ is written $F=F(L)= M^L _{f}$, where the symbol, $f \in F^2 _d \otimes \L (\H, \J)$ is defined by 
$$ f:= F (1 \otimes I_\H) = \sum _\alpha L^\alpha 1 \otimes F_\alpha ; \quad \quad F_\alpha \in \L (\H, \J). $$ In this case the operator-valued free holomorphic function $F(Z)$ takes values in $(\C ) _{nc} \otimes \L (\H , \J)$. 

We can also identify any $G \in R^\infty _d$ with its symbol:
$$ g := G1 = \sum _{\alpha \in \F ^d} G_\alpha L^\alpha 1, $$ then we can view $G$ as right multiplication by $g(Z)$,
$$ G = M^R _{g(Z)}. $$ Alternatively, we can write
$$ g = \sum _{\alpha \in \F ^d} G_{\alpha ^\dag} R^\alpha 1, $$ so that
$$ G = M^R _{g(Z)} = g^\dag (R), \quad  \mbox{where} \quad g^\dag (Z) := \sum _{\alpha \in \F ^d} G_{\alpha ^\dag} Z^\alpha. $$
That is, if $G \in R^\infty _d$ acts as right multiplication by the free NC holomorphic function $G (Z)$, then
$ M^R _{G(Z)} $ is identified with the right free Fourier series $G^\dag (R)$ (whose Ces\`aro sums converge SOT to $G$).

\begin{remark} \label{rightproduct} In the right operator-valued setting, suppose that $G (R) := g  (R) \otimes X \in R^\infty _d \otimes \L (\H , \J)$ and $F :=f  (R) \otimes Y \in R^\infty _d \otimes \L (\J , \K )$ with $f,g \in R^\infty _d$, $X \in \L (\H, \J)$, and $Y \in \L (\J , \K)$. If $H = FG$, then observe that 
$$ H ^\dag (Z) = g ^\dag (Z) f ^\dag (Z) \otimes YX. $$ This extends to a `right product' for arbitrary operator-valued free holomorphic functions on $\B ^d _\N$,  $H (Z) = F(Z) \bullet _R G(Z)$. In the scalar-valued setting this simply reduces to $F(Z) \bullet _R G(Z) = G(Z) F(Z)$. 
\end{remark}

\subsection{Operator-valued free multipliers}
\label{freedBRspace}

It will also be convenient to consider operator-valued (left and right) free multipliers between vector-valued free Hardy spaces. Namely, if $\H$ is an auxiliary Hilbert space, one can consider the NC-RKHS $H^2 (\B ^d _N ) \otimes \H$ of $\H$-valued NC functions on $\B ^d _\N$. This NC-RKHS has the operator-valued CPNC kernel:
$$ K(Z,W) \otimes I_\H, $$ and is spanned by the elements 
$$ \kz h := \kz \otimes h, \quad \quad h \in \H,$$ with inner product defined by 
$$ \ip{\kz h}{\kw g}:= \ipcn{y}{K(Z,W)[vu^*]x} \cdot \ip{h}{g}_\H, $$ for $h,g \in \H$, $Z \in \B ^d _n, W \in \B ^d _m$, $v , y \in \C ^n$ and $u, x \in \C ^m$. 
We will write $H^\infty _L (\B ^d _\N) \otimes \L (\H , \J)$ in place of $\mr{Mult} ^L (H^2 (\B ^d _\N ) \otimes \H , H^2 (\B ^d _\N ) \otimes \J )$, the $WOT-$closed \emph{left multiplier space} between these vector-valued free Hardy spaces. (That is, we write $H^\infty _L (\B ^d _\N) \otimes \L (\H , \J)$ in place of the weak operator topology closure of this algebraic tensor product). The operator-valued Schur classes, $\scr{L} _d (\H , \J), \scr{R} _d (\H ,\J)$ are then the closed unit balls of these operator-valued left and right (\emph{resp}.) multiplier spaces.

Given any bounded linear operator $T \in \L (\H , \J )$, the operator range space, $\scr{M} (T)$, is defined as the Hilbert space completion of $\ran{T}$ equipped with the inner product:
$$ \ip{Th}{Tg}_{\scr{M} (T)} := \ip{P_{\ker{T} ^\perp} h}{g} _\H. $$ Namely, this inner product is defined so that $T$ becomes a co-isometry onto its range. If $T$ is a contraction so that $I_\J-TT^*$ is positive semi-definite, one can also define the \emph{complementary space} $\scr{H} (T) := \scr{M} (\sqrt{I _\J -TT^*})$. Notice that (assuming $T$ is a contraction) $\scr{M} (T), \scr{H} (T)$ are contractively contained in $\J$. That is, for example, the embedding map $E : \scr{H} (T) \rightarrow \J$ is a contraction. In the context of Hardy space theory, if $b \in \scr{S} = [H^\infty (\D )]_1$ is a Schur class function on the disk, the complementary space $\scr{H} (b) := \scr{H} (M_b)$ is called the \emph{deBranges-Rovnyak space} of $b$. Analogously one can define vector-valued \emph{deBranges-Rovnyak spaces} for any operator-valued multi-variable (abelian or non-commutative) Schur class function for Drury-Arveson space or the free Hardy space. See \cite{FM2,Sarason-dB} for a general introduction to operator range spaces and classical deBranges-Rovnyak spaces.

If $A \in \scr{L} _d (\H , \J)$ or $B \in \scr{R} _d (\H , \J)$, consider the operator range spaces and complementary spaces $\scr{M} ^L (A) := \scr{M} (M^L _A)$, $\scr{H} ^L (A) := \scr{H} (M^L _A)$, and $\scr{M} ^R (B) := \scr{M} (M^R _B )$ and $\scr{H} ^R (B):= \scr{H} (M^R _B)$. The complementary spaces $\scr{H} ^L (A)$ and $\scr{H} ^L (B)$ are called the free left and right deBranges-Rovnyak spaces of the left and right operator-valued Schur class multipliers $A$ and $B$, respectively. All four of these operator range spaces are Hilbert spaces contractively contained in vector-valued free Hardy spaces, and are hence themselves vector-valued NC-RKHS.  For example, consider $\scr{M} ^L (A ) ,\scr{H} ^L (A)$ for $A \in \scr{L} _d (\H , \J)$. Similar to the classical setting one can show that these spaces have operator-valued CPNC kernels:
\ba K ^{\scr{M} ^L (A)} (Z,W) [P] & = & A(Z) \left(K(Z,W) [P ] \otimes I_\H \right) A(W) ^* \in \L (\J); \nn \\ & &  Z \in \B ^d _n, W \in B^d _m, \ P \in \C ^{n\times m}, \nn \ea
and 
$$ K ^{\scr{H} ^L (A)} (Z, W) = K(Z,W) \otimes I _\J - A(Z) (K(Z,W) \otimes I_\H) A(W) ^*. $$ Here, $K$ is the free Szeg\"{o} kernel.
Namely, for example, setting $K^A := K^{\scr{H} ^L (A)}$, $\scr{H} ^L (A)$ is the closed linear span of vectors of the form
$K^A \zyv{Z, y, v} g$ whose inner product is defined by:
\ba && \ip{K^A \zyv{Z,y,v} g }{K^A \zyv{W,x,u} f} _{\scr{H} ^L (A)}  \nn \\
&:= & \ip{y \otimes g}{\left(K(Z,W)[v u^*] \otimes I_\J - A(Z) (K(Z,W) [v u^*] \otimes I_\H ) A(W)^* \right) x \otimes f }_{\C ^n \otimes \J}. \nn \ea
In this vector-valued setting, for $Z \in \B ^d _n$ we write $K^A _Z : \C ^{n\times n} \otimes \J \rightarrow \scr{H} ^L (A)$ for the kernel map $K^A _Z (yv^* \otimes g) = K^A \zyv{Z,y,v} g$. 

On the other hand, if $B=B(R) \in \scr{R} _d (\H, \J)$, then $\scr{H} ^R (B)$ is spanned by the vectors 
$K^B \{ Z , v , y \} g$ with inner product: 
\ba & &  \ip{K^B \zyv{Z,y,v} g }{K^B \zyv{W,x,u} f} _{\scr{H} ^R (B)} \nn \\
&:=& \ip{y \otimes g}{K^B (Z,W) [v u^* \otimes I_\H ]   x \otimes f }_{\C ^n \otimes \J} \nn \\
K^B (Z,W)[\cdot ] & = &K(Z,W)[\cdot] \otimes I _\J  
- (K(Z,W) \otimes \mr{id} _\J) [B^\dag (Z) ( [\cdot ] \otimes I_\H ) B^\dag (W)^* ]. \nn \ea 
It is not difficult to see that free operator-valued holomorphic functions $A,B$ on $\B ^d _\N$ belong to the left or right free Schur classes if and only if the above NC deBranges-Rovnyak kernels are (completely) positive. 

\begin{remark}{ (Right Product)}
    If $F \in \mr{Mult} ^R (K_1 ,K_2) \otimes \L (\H , \J)$ is a right operator-valued multiplier between vector-valued NC-RKHS on (say) the open unit NC ball $\B ^d _\N$, then one can easily verify that for any $g \in \J$,
    $$ (M^R _F) ^* K_2 \{ Z,y,v \} g= K_1 \{ Z , y , F(Z) \bullet _R v \} g, $$ where for any $f \in \H _{nc} (K_1)$ we define:
$$ \ip{K_1 \{ Z , y , F(Z) \bullet _R v \} g}{f} =  \ip{y \otimes h}{ F(Z) \bullet _R f(Z) v}. $$
Also, in the above, any element $f$ of a $\H-$valued NC-RKHS $\H _{nc} (K)$ is such that $f(Z) \in \C ^{n\times n} \otimes \H$, so that $f(Z) v$ is to be interpreted as an element of $\C ^n \otimes \H$. 
\end{remark}

Finally, recall that any $B \in \scr{R} _d (\H , \J)$ (or in the left free Schur class) is called \emph{inner} if $B (R)  = M^R _{B^\dag}$ is an isometry. If $B$ is inner then $\scr{M} ^R (B)$ is isometrically contained in $F^2 _d \otimes \J$ and $\scr{H} ^R (B)$ is the orthogonal complement of the range of $M^R _B$ in $F^2 _d$. An operator-valued right multiplier $A \in R^\infty _d \otimes \L (\H , \J)$ is \emph{outer} if $M^R _{A^\dag (Z)}$ has dense range (equivalently if $ M^R _{A^\dag} (1 \otimes \H )$ is cyclic for the vector-valued left free shift $L \otimes I_\J$).
We will use the notation $\scr{O} ^\infty _R (\B ^d _\N)$ for the outer scalar-valued right free multipliers. 

\begin{remark}
    All results in this paper have a `left' and a `right' version, and the proofs are typically invariant under the interchange of left and right.  
\end{remark}

\section{The free Smirnov classes}

Given linear transformations $T_k : \dom{T_k} \subseteq \H \rightarrow \J$, we will use the subset notation: $T_1 \subseteq T_2$ to denote that $\dom{T_1} \subseteq \dom{T_2}$ and $T_2 | _{\dom{T_1}} = T_1$. In this case we say that $T_2$ is an extension of $T_1$ and $T_1$ is a restriction of $T_2$.

\begin{defn}
A linear transformation $T : \dom{T} \subseteq F^2 _d \rightarrow F^2 _d$ is \emph{affiliated to the right free shift} if
$$ T L \subseteq L (T \otimes I_d). $$ That is, for any $1\leq k \leq d$, $L_k \dom{T} \subseteq \dom{T}$ and $TL_k x = L_k Tx$ for any $x \in \dom{T}$. If $T$ is also closed and densely-defined, we write $T \sim R^\infty _d$.

If $T$ is affiliated to the right free shift, closed, and $1 \notin \dom{T} ^\perp$, we will write $T \lesssim R^\infty _d$.
\end{defn}

\begin{lemma} \label{freeouter1}
If $B \in R^\infty _d$ or $L^\infty _d$ is outer, then $B(Z)$ is invertible for all $Z \in \B ^d _\N$.
\end{lemma}
\begin{proof}
By definition, any $B \in L^\infty _d$ or $R^\infty _d$ is outer if it has dense range and this is equivalent to $B(L) ^*$ or $B(R) ^*$ (\emph{resp.}) being injective. If $B\in L^\infty _d$, and there is a $Z \in \B _n ^d$ so that $B(Z) \in \C ^{n\times n}$ is not invertible, then we can find a $y \in \C ^n$ so that $B(Z) ^* y =0$. In this case $B ^* \kz =0$ for any non-zero $v ^* \in \C ^n$, so that $B(L) ^*$ is not injective and $B$ is not outer.
\end{proof}

\begin{lemma} \label{freeouter}
If $B \in \scr{R} _d$ or $\scr{L} _d$ and $B(Z)$ is strictly contractive on $\B ^d _\N$ then $1 -B$ is outer.
\end{lemma}
This is a free analogue of \cite[Lemma 2.3]{Smirnov}.
\begin{proof}
Assume that $B \in \scr{L} _d$ is strictly contractive on the NC unit ball. It follows that $(1-B) (Z)$ is invertible on $\B _\N ^d$. Suppose that there is a $f \in H^2 (\B ^d _\N )$ so that $(I -B(L) ^* ) f = 0$. Since $B(L)$ is a contraction, it follows that
$B(L) f = f$, so that $(I - B(L) ) f =0$.  It then follows that
$$ (I -B(Z)) f(Z) = 0 ; \quad \quad Z \in \B ^d _\N, $$ and since $1-B$ is invertible on the NC unit ball, $f \equiv 0$. This proves that $(I-B(L) ^*)$ is injective so that $I-B$ is left outer.
The right version of the proof is analogous.
\end{proof} 
\begin{remark} \label{strict}
Note that if $B _\emptyset = B(0) =0$ and $B$ is Schur, then $\| B(Z) \| \leq \| Z \|$ for any $Z$ in the NC open unit ball by the free Schwarz lemma \cite[Theorem 2.4]{Pop-freeholo}. Combining this fact with automorphisms of unit ball of $\C ^{n \times n}$, \emph{i.e.} operator-M\"{o}bius transformations \cite{MMdBR}, shows that if $|B _\emptyset | <1$ then $B(Z)$ is strictly contractive on $\B _\N ^d$.
\end{remark}
\begin{defn}\label{def:free-smirnov}
The \emph{left and right free Smirnov classes} are the sets of all free (NC) functions on the open NC unit ball, $\B _\N ^d$, which can be written as ratios of $H^\infty _{L} (\B _\N ^d) \simeq L^\infty _d$ or $H^\infty _{R} (\B _\N ^d)$ functions with (respectively) left or right outer denominator:
\ba & & \scr{N} _d ^+ (R) := \left\{ \left. A ^\dag (Z) ^{-1} B ^\dag (Z) \right| \ B \in H^\infty _R (\B _\N ^d), \ A \in \scr{O} ^\infty _R (\B ^d _\N )  \right\}, \nn \\
 &  & \scr{N} _d ^+ (L) := \left\{ \left. B(Z) A(Z) ^{-1}  \right| \ B \in H^\infty _L (\B _\N ^d ), \  A \in \scr{O} ^\infty _L (\B ^d _\N ) \right\}. \nn \ea
\end{defn}

(The previous lemma, Lemma \ref{freeouter1}, implies that any outer $A \in R^\infty  _d$ or $L ^\infty _d$ is invertible on the NC unit ball.) Our goal is to identify the set of all $T \sim R^\infty _d$ as unbounded right multipliers with symbols in the Smirnov class $\scr{N} _d ^+ (R)$. Although our main interest is in densely-defined operators, we will also obtain several results on $T \lesssim R^\infty _d$.
\begin{remark}
One could also define the \emph{right free Nevanlinna class}, $\scr{N} _d (R)$, as the set of all free functions on the NC unit ball which can be written as `ratios' of $H^\infty _R (\B ^d _\N)$ functions with $\ran{B^\dag (Z)} \subseteq \ran{A^\dag (Z)}$ for all $Z$ in a matrix-norm dense set in $\B ^d _\N$ (on each level). While we will not pursue this, we expect there are analogous results relating $\scr{N} _d (R)$ and $T \lesssim R^\infty _d$. 
\end{remark}
\begin{defn}
    We say that a closed $T:\dom{T} \subseteq F^2 _d \rightarrow F^2 _d$ affiliated to $R^\infty _d$ is \emph{local} if
    $$ L^* \left( \ran{L} \cap \dom{T} \right) \subseteq \dom{T} \otimes \C ^d. $$ That is, $T$ is local if, whenever $x \in \dom{T}$ and $x = L L^* x$,
then $L_k ^* x \in \dom{T} \ \forall 1 \leq k \leq d$.
\end{defn}
\begin{lemma}{ (\cite[Lemma 2.4]{JM-H2Smirnov})} \label{BabySmirnov}
Let $T(Z)$ be a free holomorphic function on $\B _\N ^d$. Define the domain
$$ \dom{T} := \{ f \in H^2 (\B _\N ^d ) | \ f(Z) T(Z) \in H^2 (\B _\N ^d ) \}. $$ Then the linear transformation of right multiplication by $T(Z)$, $T := M^R _{T(Z)}$, with domain $\dom{T}$ is closed and affiliated to the right free shift. This $T$ is also local.
\end{lemma}
The original lemma statement above from \cite{JM-H2Smirnov} assumed that $T$ was densely-defined, but this assumption is not needed.
\begin{cor} 
    If $T(Z) := A(Z) ^{-1} B(Z) \in \scr{N} _d ^+ (R)$ is a right-free Smirnov function then $T:= \ran{M^R _{T(Z)}}$, defined on its maximal domain, is densely-defined, closed, and local.
\end{cor}
\begin{proof}
    The domain of $T$ includes the range of $M^R _{A(Z)} = A^\dag (R)$, which is dense since $A^\dag (R)$ is right-outer, by assumption. 
\end{proof}
\begin{defn} \label{RSmirnovdef}
We say that a closed linear transformation $T: \dom{T} \subseteq F^2 _d \rightarrow F^2 _d$ is \emph{right-Smirnov} if $\dom{T}$ is dense and $T$ acts as right multiplication
by a right free Smirnov function $T(Z) = A(Z) ^{-1} B(Z) \in \scr{N} _d ^+ (R)$: For any $x \in \dom{T}$,
$$ (Tx) (Z) = x(Z) A(Z) ^{-1} B(Z). $$  If $T$ is right-Smirnov, we write $T \sim \scr{N} _d ^+ (R)$.
\end{defn}
If $T = M^R _{T(Z)}$ is right-Smirnov, and acts as right multiplication by the right free Smirnov function $T(Z)= A(Z) ^{-1} B(Z)$, then we write $T = T^\dag (R) = B^\dag (R) A ^\dag (R) ^{-1}$ or $T = M^R _B (M^R _{A} ) ^{-1} = M^R _B M^R _{A^{-1}}$.
\begin{remark} It is not difficult to see that the left and right
Smirnov classes are distinct. Indeed, the free function
$$ H(Z_1, Z_2) = (1-Z_2)^{-1}Z_1, $$  is evidently right Smirnov ($1-Z_2$ is right-outer, by Lemma \ref{freeouter}). However, it cannot correspond to a
densely defined {\em left} multiplication operator, and hence cannot
be left Smirnov. Indeed, if a nonzero function
$$  F=\sum_\alpha  f _\alpha L^\alpha 1,  $$ belongs to $F^2_2$, consider the free function
$$ H(Z)F(Z) = (1-Z_2)^{-1} Z_1 F(Z). $$  
Choose any word $\alpha$ for which $f_\alpha \neq 0$. Then for each $n\geq 1$,
the coefficient of the word $2^n1 \alpha$ in the expansion of $HF$ is
precisely $f_\alpha$, and hence the coefficients of this series are
evidently not square-summable. 
\end{remark}
\begin{lemma} \label{babySmirnov}
    If $T(Z)$ is a free holomorphic function and $T:= M^R _{T(Z)} : \dom{T} \rightarrow F^2 _d$ is densely-defined, then for any $Z \in \B _\N ^d$, $T^* :\dom{T}^* \rightarrow \mc{D} _T$ preserves the range of $K \{ Z, y , \cdot \}$:
    $$ T^* K_Z (yv^* ) = T^*  \kz = K \zyv{ Z, y , T(Z) v} = K_Z (yv^*T(Z) ^*).$$ In particular, $\dom{T^*}$ contains the linear span of the point evaluation vectors
    $$\bigvee _{Z \in \B  ^d _n; \ v,y\in \C ^n} \kz. $$
\end{lemma}
\begin{proof}
    Same easy calculation as in the bounded case: for any $F \in \dom{T}$,
$$ \ip{\kz}{TF} = \ipcn{y}{F(Z)T(Z) v} = \ip{K \zyv{ Z,y,T(Z)v}}{F}. $$
\end{proof}
\begin{cor} \label{corelem}
    If $T \sim \scr{N} _d ^+ (R)$ is right-Smirnov and defined on its maximal domain, then $\bigvee _{Z, v,y} \kz$ is a core for $T^*$.
\end{cor}

\begin{proof}
This follows from the maximality assumption on the domain of $T$. By definition of right-Smirnov, Definition \ref{RSmirnovdef}, and Lemma \ref{BabySmirnov}, $T$ has no non-trivial extensions which act as multiplication by $T(Z) \in \scr{N} _d ^+ (R)$.
Let $T_0 ^*$ be the closure of the restriction of $T^*$ to $\bigvee \kz$. Then $T_0 ^* \subseteq T^*$ is densely-defined and closed so that its adjoint,
$T_0 \supseteq T$ is densely-defined and closed. But,
$$ T_0 ^* \kz = K \zyv{Z, y, T(Z)v} = T^* \kz, $$ so that $T_0$ acts as right multiplication by $T(Z)$ on its domain. By Lemma \ref{BabySmirnov} $T_0 = T$.
\end{proof}

The following lemma shows that if $T \sim \scr{N} _d ^+ (R)$, then all free polynomials $\C \{ Z_1 , ... , Z _d \}$ belong to the domain of $T^*$ (by showing that, in fact, any free polynomial $p(L)1$ is equal to some $\kz$ for a certain choice of jointly nilpotent $Z$):

\begin{lemma}
    For any $\alpha \in \F ^d$ there exist $n\in \N$, $Z \in \B ^d _n $ and $v^*,y \in \C ^n$ so that $\kz = L^\alpha 1$.
\end{lemma}
\begin{proof}
Fix $\alpha$, and let $\mathcal H$ be the finiite-dimensional subspace of $F^2_d$ spanned by the basis vectors $e_\beta = L^\beta 1$ for which $|\beta|\leq |\alpha|$. We let $\Lambda_1, \dots, \Lambda_d$ denote the compressions of the shifts $L_1, \dots, L_d$ to the space $\mathcal H$. Fix now $0<r<1$. For any function $F\in F^2_d$, the function $F_r (Z):=F(rZ)$ is expressible as a norm-convergent power series in $L$ and therefore belongs to $L^\infty_d$. Moreover, writing $F_r=\sum_{\gamma\in \mathbb F^d} f_\gamma r^{|\gamma|} L^\gamma$ we have by construction
\begin{equation}
r^{|\alpha|} f_\alpha = \langle e_\alpha, F_r(L)e_{\varnothing} \rangle_{F^2_d} = \langle  e_\alpha, F_r(\Lambda) e_{\varnothing}\rangle_{\mathcal H} 
\end{equation}
and so if we put $y=r^{-|\alpha|}e_\alpha, u=e_\varnothing$, and $Z=r^{|\alpha|}\Lambda^\alpha$, we obtain $f_\alpha = \langle F, \kz\rangle$. 
\end{proof}

Let $Z(\alpha) , y_\alpha, v_\alpha$ be such that $K \{ Z(\alpha ) , y _\alpha , v_\alpha \} (Z) = Z^\alpha \simeq L^\alpha 1$. 
\begin{cor}
Suppose that $p(Z) = \sum _\alpha Z^\alpha p_\alpha$ is any free polynomial, and set $p:= p(L) 1 \in F^2 _d$. Then, for any fixed $0<r<1$, one has $p = \kz$ with $Z := \bigoplus Z(\alpha)$, $y = \bigoplus y_\alpha$, and $v^* := \bigoplus v^* _\alpha $, where $Z(\alpha) = r C (\alpha) \in \B ^d _{|\alpha| +1}$,  $y_\alpha = \frac{1}{r^{|\alpha| } }\frac{p_\alpha}{|p_\alpha | ^{1/2}} e_1  \in \C ^{|\alpha| +1}$, and $v^* _\alpha := |p _\alpha | ^{1/2} e_{|\alpha| +1} \in \C ^{|\alpha | +1}$.
\end{cor}

\begin{remark}
As before, one can apply the above corollary to prove that free polynomials, $\bigvee L^\alpha 1 = \C \{ L_1 , ... , L_d \}$, are a core for the adjoint of any right-Smirnov $T$, defined on its maximal domain. 
\end{remark}

\section{Operators affiliated to the right free shift}

The main goal of this section is to prove that any $T \sim R^\infty _d$ is such that $T \sim \scr{N} _d ^+ (R)$. That is, any closed, densely-defined linear operator affiliated to the right free shift acts as right multiplication by a right free Smirnov function.

\subsection{Some general facts concerning closed operators} \label{closedops}

Let $A : \dom{A} \subseteq \H \rightarrow \J$ be a linear transformation. Recall that $A$ is called \emph{closed} if its graph, 
$$ G(A) := \{ h \oplus Ah | \ h \in \dom{A} \} \subseteq \H \oplus \J,$$ is a closed subspace. 
A linear subspace $\mc{D} \subseteq \dom{A}$ is called a \emph{core} for a closed linear transformation $A$ if $A = \ov{A | _{\mc{D}}}$, \emph{i.e.} the minimal closed extension of $A | _{\mc{D}}$ is $A$.
A densely-defined linear transformation $A$ has closed extensions if and only if its adjoint, $A^*$, is densely-defined. (Further recall that the adjoint is always a closed linear transformation.) In this case if $A$ is closed and $J$ is the anti-idempotent unitary:
$$ J:= \bbm 0 & -1 \\ 1 & 0 \ebm : \H \oplus \J \rightarrow \J \oplus \H, $$ then
$$ G(A^*) = (J G(A) ) ^\perp = J G(A) ^\perp. $$ 

Consider the following theorem of von Neumann (see, for example, \cite[Theorem 5.1.9]{AnalysisNow}):

\begin{thm} \label{vNthm}
    Let $A : \dom{A} \subseteq \H \rightarrow \J$ be a closed, densely-defined linear transformation. Then $A^* A$ is a closed densely-defined positive self-adjoint operator
on the domain: $$ \dom{A^* A} = \{ h \in \dom{A} | \ Ah \in \dom{A^*} \}. $$
Moreover, for any $h \in \H$ the extremization problem
$$ D^2 _h := \inf _{g \in \dom{A} }  \left( \| h -g \|  ^2 + \| A g\| ^2 \right)$$ has a unique solution $g_h \in \dom{A^* A } \subset \dom{A}$ and
$$ (I + A ^* A ) g_h = h. $$ In particular $I +A ^* A$ is a positive bijection from $\dom{A^*A}$ to $\mathcal H$, and $\Delta _A :=(I +A^* A) ^{-1}$ is a positive contraction. The linear domains $\dom{A^*A}$ and $A\dom{A^*A}$ are cores for $A, A^*$, respectively.
\end{thm}
\begin{defn} \label{qadef}
Let $A$ be as in Theorem~\ref{vNthm}. Define $X : G(A) \rightarrow \H$ and $Y: \H \rightarrow G(A)$, by
$$ X (h \oplus Ah ) := h; \quad \quad \mbox{and} \quad Y h := \Delta _A ^{-1} h \oplus A \Delta _A ^{-1} h. $$
\end{defn}

\begin{lemma} \label{qalemma}
Both $X, Y$ are quasi-affinities and $Y = X^*$.
\end{lemma}

\begin{proof}
First, $X$ is clearly contractive. It is injective since $G(A) \subseteq \H \oplus \J$ is a graph, it has dense range since we assume $A$ is densely defined (in $\H$). This proves $X$ is a quasi-affinity.

For any $x \in \H$, consider
\ba \| Y x \| ^2 &= & \| \Delta _A ^{-1} x \oplus A \Delta _A ^{-1} x \| ^2 \nn \\
	& = & \ip{ (I +A ^* A ) \Delta _A ^{-1} x }{ \Delta _A ^{-1} x } \nn \\
	& = & \ip{x}{\Delta _A ^{-1} x } \nn \\
	& \leq & \| \Delta _A  ^{-1} \| \| x \| ^2. \nn \ea
This shows that $Y$ is bounded, and also we see that $Yx =0$ if and only if $\ip{x}{\Delta _A ^{-1} x } =0$. Since $\Delta _A ^{-1}$ is positive and injective, this happens if and only if $x =0$, so that $Y$ is injective.
Also, by definition $\ran{Y} = \bigvee \{ \Delta _A ^{-1} x \oplus A \Delta _A ^{-1} x | \ x \in \H \}$, and this is dense in $G(A)$ since $\ran{\Delta _A } ^{-1} = \dom{A^* A}$ is a core for $A$. This proves $Y$ is a quasi-affinity.

Finally, given any $y \oplus Ay \in G(A)$ and any $x \in \H$, consider
\ba \ip{x}{Y^* (y \oplus Ay ) } & = & \ip{Yx}{y \oplus Ay} \nn \\
& = & \ip{ \Delta _A ^{-1} x \oplus A \Delta _A ^{-1} x }{ y \oplus Ay} \nn \\
& = & \ip{ (I +A^*A) \Delta _A ^{-1} x }{y} \nn \\
& = & \ip{x}{y}. \nn \ea This proves that $Y^* (y \oplus Ay ) = y = X (y \oplus Ay)$, so that $Y ^* = X$.
\end{proof}

\begin{remark} \label{Toepdom}
    Note that $\ran{XX^*} = \ran{\Delta _A ^{-1}} = \dom{I +A^*A}=\dom{A^*A}$.
\end{remark}

\subsection{The wandering space for $G(T)$} 

Let $T:\dom{T} \subseteq F^2 _d \rightarrow F^2 _d$ be any closed linear transformation affiliated to the right free shift. (We do not assume that $T$ is densely defined.) As is straightforward to check, a closed linear transformation, $T$, is affiliated to $R^\infty _d$ if and only if its graph, $G(T)$, is an invariant subspace for $L\otimes I_2$ so that $L_T := (L\otimes I_2 )| _{G(T)}$ is a (row) isometry. Recall that by the Popescu-Wold decomposition of row isometries, any row isometry $\Pi : \H \otimes \C^d \rightarrow \H$ decomposes as the direct sum of several copies of $L$ and a Cuntz unitary (an onto row isometry) \cite{Pop-dilation}. Since $L$ is pure (\emph{i.e.} it has no Cuntz unitary direct summand), it follows from \cite{Pop-dilation} that $L_T$ is also pure. Recall that a vector $h \in \H$ is called \emph{wandering} for a row isometry $\Pi : \H \otimes \C ^d \rightarrow \H$ (or wandering for $G(\Pi)$) if

 $$ \ip{\Pi ^\alpha h}{\Pi ^\beta h} = \delta _{\alpha, \beta } \| h \| ^2. $$
(That is, $\{\Pi^\alpha h: \alpha\in \mathbb F^d\}$ is an orthogonal set, and orthonormal if $\|h\|=1$. )
By \cite[Theorem 1.3]{Pop-dilation}, every non-zero vector in the \emph{wandering space}, \be \mc{W} (\Pi) = \H \ominus \Pi (\H \otimes \C ^d ), \label{wanddef} \ee
is wandering, and the span of the (orthogonal) subspaces $\Pi^\alpha \mathcal W$, as $\alpha$ ranges over $\mathbb F^d$, is dense in $\mathcal H$. 
Let $\{ \theta _k \} _{k=1} ^N$, where $N \in \N \cup \{\infty \}$, be an orthonormal basis of wandering vectors for $G(T)$.  For a wandering vector $w\in \mathcal W(T)$ let
$$ F^2(w) :=\bigvee (L^\alpha \otimes I_2 ) w ) ^{-\| \cdot \| }$$   It is easy to see that if $w_1 \perp w_2 \in \mc{W} (T)$ are orthogonal wandering vectors then 
$$ F^2 (w_1) \perp F^2 (w_2). $$ The Popescu-Wold decomposition further implies that
$$ G(T) = \bigoplus _{k=1} ^N F^2 (\theta _k), $$ is the direct sum of mutually orthogonal cyclic invariant subspaces for $L \otimes I_2$ with mutually orthogonal cyclic unit wandering vectors $\theta _k$. 

\begin{lemma}
For any $1 \leq k \leq N$, right multiplication by $\theta _k =: a_k \oplus b_k \in F^2 _d \otimes \C ^2$, $M^R _{\theta _k}: F^2 _d \rightarrow F^2 _d \otimes \C ^2$ is an isometry, \emph{i.e.} $$ \Theta _k (R) := M^R _{\theta _k} = M^R _{\bbm a_k \\ b_k \ebm} =: \bbm A_k (R) \\ B_k (R) \ebm $$ is a two-component inner right multiplier with scalar component right multipliers $A_k (R), B_k (R) \in \scr{R} _d$. 
\end{lemma}

\begin{proof}
This is easy to see: 
\ba
\ip{\Theta _k (R) L^\alpha 1 }{\Theta _k (R) L^\beta 1}_{F^2} & = & \ip{L^\alpha \otimes I_2 \theta _k }{L^\beta \otimes I_2 \theta _k} _{F^2 \oplus F^2} \nn \\
& = & \delta _{\alpha ,\beta}, \nn \ea since $\theta _k$ is a unit wandering vector. It follows that $M^R _{\theta _k}$ extends by continuity to an isometry of $F^2 _d$ into $G(T) \subseteq F^2 _d \otimes \C ^2$ which intertwines $L$ and $L\otimes I_2$. 
\end{proof} 
\begin{remark}
In general, since $\ran{\Theta _k (R) } \subseteq G(T)$, we always have that 
$$ A_k (R) \equiv 0 \quad \Rightarrow \quad B_k (R) \equiv 0, $$ and this happens if and only if $b_k =B_k (R) 1$, $a_k = A_k (R) 1$ are zero in $F^2 _d$. 
Since we assumed that each $\theta _k := \bbm a_k \\ b_k \ebm$ is a unit wandering vector, this cannot happen.
\end{remark}
Define $\Theta   (R) : F^2 _d \otimes \C ^N \rightarrow F^2 _d \otimes \C ^2$ by
\ba \Theta  (R) &:=& \bbm \left. \Theta _1 (R) \  \right|  & \Theta _2 (R) \left. \right|  & \cdots \ &  \left| \Theta _N (R)  \right. \ebm, \nn \\
& = & \bbm A(R) \\ B(R) \ebm,  \label{charfun} \ea where $A(R), B(R) : F^2 _d \otimes \C ^N \rightarrow F^2 _d$ are defined by
\be B(R) := \bbm B_1 (R), & B_2 (R), & \cdots \ebm, \quad \quad A(R) := \bbm A_1 (R), & A_2 (R), & \cdots \ebm, \label{ABdef} \ee and recall that $N \in \N \cup \{ \infty \}$ could be infinite. If $T$ is densely-defined note that $\ran{A(R)}$ must be dense in $F^2 _d$, and observe that 
\be G(T) = \ran{\Theta  (R)} = \bigoplus _{k=1} ^N \ran{\Theta _k (R) }. \label{graphdecomp} \ee Since each $\Theta _k (R)$ is an isometry and the $\Theta _k$ have mutually orthogonal ranges, it follows that $\Theta (R)$ is also an isometry, \emph{i.e.}, an inner right multiplier. 

If $T \sim \scr{N} _d  (R)$ acts as right multiplication by some $A(Z) ^{-1} B(Z)$ on its domain, then it is clearly injective. Conversely, if $T \sim R^\infty _d$, set 
$$ \Gamma _0 (T) := G(T) \cap (F^2 _d \oplus 0) = \{ x \oplus 0 \in G(T) | \ x \in \ker{T} \subseteq \dom{T} \}.$$ It follows that $T$ will be injective if and only if $\Gamma _0 (T) = \{ 0 \}$. 

\begin{lemma} \label{injLemma}
   If $T \sim R^\infty _d$ then it is injective.
\end{lemma}

\begin{proof}
Suppose that $x \oplus 0 \in \Ga _0 (T)$. Then we can write 
$$ x \oplus 0 = \bigoplus x_k \oplus Tx_k, $$ where each $x_k \oplus Tx_k \in \ran{\Theta _k (R)}$. Fix a value of $j \in \{ 1 , ... , N \}$ and re-write this decomposition as:
$$ x \oplus 0 = (x_j \oplus Tx_j) - (y \oplus Ty), $$ with $x_j \oplus Tx_j \perp y \oplus Ty$, $x = x_j -y$, and $Tx_j = T y$.
Hence, 
\ba \| x_j - y \| ^2 & = &  \| x _j \| ^2 - 2 \re{\ip{x_j}{y}} + \| y \| ^2 \nn \\
& = & \| x _j \| ^2 +2 \re{\| Ty \| ^2} + \| y \| ^2 \nn \\
& \geq & 2 \| T y \| ^2 +  \| x_j - y \| ^2, \nn \ea  by the triangle inequality. This proves that $Ty=Tx_j =0$, so that in fact, $x_j \oplus Tx_j = A_j (R) h \oplus B_j (R) h \in \ran{\Theta _1 (R) }$ is such that $ B_j (R) h =0$. If $B_j \in R^\infty _d$ is not identically zero, then it is injective (any non-zero right multiplier is injective \cite[Theorem 1.7]{DP-inv}) and $h=0$ so that $x_j \oplus Tx_j = A_j (R) h \oplus B_j (R) h = 0$. It follows that $x = \bigoplus x_k \oplus 0$, where if $x_k \neq 0$ then $A_k$ is inner and $B_k \equiv 0$. However, it then follows that 
$$ \ran{A_k (R)} \perp \bigvee _{j \neq k} \ran{A_j (R) } \subset \dom{T}, $$ so that
$$\bigvee _{j \neq k} \ran{A_j (R) }, $$ is a dense $L-$invariant subspace of the co-invariant subspace $\ran{A_k (R)} ^\perp$. It follows that $\ran{A_k (R)} ^\perp$ is both $L-$invariant and co-invariant (and hence $L-$reducing). Since $L$ has no non-trivial reducing subspaces we conclude that $\ran{A_k (R)} = \{ 0 \}$ and $A_k (R) \equiv 0$. This contradicts the assumption that each $A_k \neq 0$ and we conclude that $T$ is injective.
\end{proof}

\subsection{A particular wandering vector}

As in Subsection \ref{closedops}, define the linear map $X : G(T) \rightarrow F^2 _d $ by
$$ X (F \oplus TF ) := F \in F^2 _d. $$ This $X$ is necessarily injective, $\ran{X} = \dom{T}$, and if we view $X$ as a map
into $\mc{D} _T := \dom{T} ^{-\| \cdot \|}$, then $X$ is a quasi-affinity. In this section we assume that $T \lesssim R^\infty _d$, so that $1 \notin \dom{T} ^\perp$, and, in particular, $X^*1 \neq 0$. (Some of the arguments in this subsection appear already in \cite{JM-H2Smirnov}, we reproduce them here for convenience.)

As is easily verified:
\begin{lemma}
$X (L \otimes I_2) | _{G(T) \otimes \C ^d} = L (X \otimes I_d).$
\end{lemma}

\begin{lemma} \label{wandvector}
$X^* 1$ is wandering for $L \otimes I_2 | _{G(T) \otimes \C ^d}.$
\end{lemma}
\begin{proof}
    Consider
$$ \ip{ L^\alpha \otimes I_2 \ X^*1}{ L^\beta \otimes I_2 \ X^* 1}. $$ This vanishes if $\beta, \alpha$ are not comparable,
so assume without loss of generality that $\alpha = \beta \ga$. Then this evaluates to:
$$ \ip{ X L^\ga \otimes I _2 X ^* 1}{1} = \ip{L^\ga XX^* 1}{1} = \delta _{\ga , \emptyset} \| X^* 1 \| ^2. $$
\end{proof}
\begin{remark} \label{orthoX}
The orthogonal complement of  $ X^* 1$ is:
\be (X^* 1 ) ^{\perp} = \{ L \mbf{F} \oplus T L \mbf{F} | \ L \mbf{F} \in \dom {T} \} \label{wandperp}. \ee
On the other hand,
$$ \ran{L \otimes I_2 | _{G(T)}} = \mathcal W(T)^\bot\subseteq \{X^*1\}^\bot.$$

%begin{align*} &  \ran{L \otimes I_2 | _{G(T)}} = \\
%& \{ L \mbf{G} \oplus TL \mbf{G} = L \mbf{G} \oplus L (T \otimes I_d ) \mbf{G} | \ \mbf{G} \in \dom{T} \otimes \C ^d \}  \subseteq   (X^* 1 ) ^\perp. \end{align*} 

By Definition \ref{qadef} and Lemma \ref{qalemma},
$$ X^*1 = \Delta _T ^{-1} 1 \oplus T \Delta _T ^{-1} 1. $$
It follows that
$$ \| X ^* 1 \| ^2 = \ip{X^*1}{X^*1} = \ip{X ( \Delta _T ^{-1} 1  \oplus T \Delta _T ^{-1} 1 ) }{1} = \ip{ \Delta _T ^{-1} 1}{1},$$
and \be w _1   := \frac{X^*1 }{\| X^* 1 \| } = \frac{X^*1}{\sqrt{\ip{1}{\Delta _T ^{-1} 1}}}, \label{thewv} \ee is a unit norm wandering vector for $L \otimes I_2 | _{G(T)}$. For the remainder of the paper we choose $\theta _1 = w_1$ as the first unit wandering vector in a choice of orthonormal basis for the wandering space $\mc{W} (T)$. 
\end{remark}

\begin{cor} \label{Smircor}
If $T \lesssim R^\infty _d$ then the following are equivalent:
\bn
    \item $T$ is local (that is, $\dom{T} \cap \ran{L}$ is $L^*$-invariant).
    \item $X^* 1$ spans the wandering space for $G(T)$.
\en
If $T \sim R^\infty _d$ then the above two conditions are equivalent to:
\bn
    \item[(3)] $T \sim \scr{N} _d ^+ (R)$.
\en
\end{cor}

\begin{proof}
Suppose that $F \oplus TF \perp  X^* 1$. As observed above in Remark \ref{orthoX}, this means that $F = L \mbf{F}$ for some $\mbf{F} \in F^2 _d \otimes \C ^d$.
Since $F \in \ran{L}$, the assumption that $T$ is local would imply that $\mbf{F} = L^* F \in \dom{T} \otimes \C ^d$.
However,
\ba F \oplus TF & = & L\mbf{F} \oplus T L \mbf{F} \nn \\
& = & (L \oplus L ) \left( \mbf{F} \oplus (T\otimes I_d ) \mbf{F} \right) \in \ran{ L \otimes I_2 | _{G(T)}} = \mc{W} (T) ^\perp. \nn \ea It follows that the wandering space, $\mc{W} (T)$, is spanned by $X^*1$. 

Conversely if $X^*1$ spans the wandering space of $G(T)$ then 
$$ G(T) = \ran{\bbm A(R) \\ B(R) \ebm }, $$ where $A, B \in \scr{R} _d$ are defined by $X^*1 / \| X^* 1 \| =: A(R) 1\oplus B(R) 1$ and $\dom{T} = \ran{A(R)}$. $T$ must be local in this case: If $F \in \ran{L} \cap \dom{T}$, then there is a $\mbf{F} \in F^2 _d \otimes \C ^d$ so that
$$ F \oplus TF = L \mbf{F} \oplus T L \mbf{F} = A(R) x \oplus B(R) x. $$ Since $T \lesssim R^\infty _d$, $1$ does not annihilate $\dom{T}$, and this is equivalent to $1 \notin \ker{A(R) ^*}$, \emph{i.e.}, $A_\emptyset = A(0) \neq 0$. This, and the fact that $L \mbf{F} = A(R) x \in \ran{L}$ implies that $x = L \mbf{x} \in \ran{L}$ for some $\mbf{x} = (\mbf{x} _1, ... , \mbf{x} _d ) \in F^2 _d \otimes \C ^d$. Then, 
$$ F \oplus TF = L (A(R) \otimes I_d) \mbf{x} \oplus L (B(R) \otimes I_d) \mbf{x}, $$ so that for each $1 \leq k \leq d$, 
$$ A(R) \mbf{x} _k \oplus B(R) \mbf{x} _k = (L_k ^* \otimes I_2) F \oplus TF \in G(T). $$ This proves that if $F \in \dom{T} \cap \ran{L}$, then $L_k ^* F \in \dom{T}$ for any $1\leq k \leq d$, \emph{i.e.}, $T$ is local.

If also $T \sim R^\infty _d$ is densely-defined, then $A(R)$ is necessarily right-outer, so that $A ^\dag (Z)$ is invertible on the NC unit ball by Lemma \ref{freeouter1}, and $T = M^R _{T(Z)}$ with $T(Z)= A^\dag (Z) ^{-1} B^\dag (Z) \in \scr{N} _d ^+ (R)$. Any such $T$ is necessarily local, by Lemma \ref{BabySmirnov}.
\end{proof}

\begin{lemma} \label{asmult}
    Given $T \sim R^\infty _d$, $T$ acts as right multiplication by a free holomorphic function, $T(Z)$, if and only if $\dom{T}$ contains an $L-$cyclic vector.
\end{lemma}

\begin{proof}
If $T$ acts as right multiplication by $T(Z)$, then $T$ is local by Lemma \ref{babySmirnov}, so that $X^*1$ spans the wandering space of $G(T)$ by Corollary \ref{Smircor}. It is clear then that $XX^*1$ is $L$-cyclic since $T$ is assumed to be densely-defined.

Conversely if $\dom{T}$ contains an $L-$cyclic vector, $h \in F^2 _d$, then, by \cite[Theorem 1.4]{JM-H2Smirnov}, there exist $u,v\in F^\infty_d$, with $v$ outer, so that $h(Z)=u(Z)v(Z)^{-1}$.

 Since $h$ is assumed to be $L$-cyclic, $u$ must also outer, and hence $u(Z)$ and $v(Z)$ are both invertible for all $Z \in \B _\N ^d$ by Lemma \ref{freeouter1}.  It follows that $h(Z)$ is invertible for any $Z \in \B ^d _\N$. Define the free holomorphic function:
$$ T(Z) := h(Z) ^{-1} (Th) (Z). $$ Then,
$$ (TL^\alpha h) (Z) = Z^\alpha (Th) (Z), $$ and
$$ (L^\alpha h)(Z) T(Z) = Z^\alpha h(Z) h(Z) ^{-1} (Th)(Z) = Z^\alpha (Th) (Z) = (T L^\alpha h) (Z), $$ and this proves the claim since $h$ is $L-$cyclic.
\end{proof}

\subsection{Free reproducing kernel analysis} \label{rkanalysis}

In this subsection, let $T:\dom{T} \subseteq F^2 _d \rightarrow F^2 _d$ be any closed linear transformation affiliated to the right free shift. As before, $\mc{D} _T := \dom{T} ^{-\| \cdot \| } \subseteq F^2 _d$, and
$$ G(T) = \ran{\Theta (R)} \subseteq F^2 _d \otimes \C ^2, $$ is a $L \otimes I_2 -$invariant subspace. Also as before, we decompose $G(T)$ into the cyclic $L\otimes I_2$-invariant subspaces
$$ \ran{\Theta (R)} = \bigoplus \ran{\Theta _k (R)}, \quad \quad \Theta _k = \bbm A_k \\ B_k \ebm. $$ 

The graph of the adjoint, $T^*$, is then $G(T^* ) = \ran{J \Theta (T) } ^\perp,$ where $J : \mc{D} _T \oplus F^2 _d \rightarrow F^2 _d \oplus \mc{D} _T$ is the unitary from Subsection \ref{closedops}. In particular,
$$ J \Theta = \bbm -B \\ A \ebm  \in \scr{R} _d (\H , \C ^2 ),  \quad \mbox{is also inner.} $$ 
In the above, recall that we chose $\H = \C ^N$ with $N \in \N \cup \{ \infty \}$ the number of unit wandering vectors in an orthonormal basis for the wandering space of $G(T)$.

Consider the four linear maps, $X, Y : G(T) \rightarrow F^2 _d$, $X_*, Y_* : G(T^*) \rightarrow F^2 _d$, where $X,X_*$ are projection onto the first component and $Y,Y_*$ are projection onto the second component. It follows that
\begin{align*} \ran{X} &= \dom{T}, \ \ran{Y} = \ran{T}, \quad \mbox{and}, \\
\ran{X_*} &= \dom{T^*}, \ \ran{Y_*} = \ran{T^*}. \end{align*} Note that $X, X_*$ are injective, and if $T\sim R^\infty _d$, $Y$ is also injective by Lemma \ref{injLemma}, but $Y_*$ need not be injective. 

Define the four operator-range spaces: $\mc{D} := \scr{M} (X), \ \mc{R} := \scr{M} (Y)$,  \mbox{and} \nn $\mc{D} _* := \scr{M} (X_*)$, and $\mc{R} _* := \scr{M} (Y_* )$. Recall that, $\mc{D} = \ran{X} = \dom{T}$, and similarly $\mc{R} = \ran{T}$, $\mc{D} _* = \dom{T^*}$ and $\mc{R} _* = \ran{T^*}$, as vector spaces. However the operator-range space inner product is defined so that each of $X, Y, X_*, Y_*$ are co-isometries onto their range spaces, see Subsection \ref{freedBRspace}. We will write $U, V , U_*, V_*$ for the respective isometries from $\mc{D}, \mc{R}, \mc{D} _*$ and $\mc{R} _*$ into $G(T)$ and $G(T^*)$. This means, for example, that $U  ^*$ acts as $X$, projection onto the first component, and similarly the adjoints of the other three isometries are projection onto the first or second component. Since $X$ and $X_*$ are injective, it follows that $U,  U_*$ are onto isometries while the range of $V$ and $V_*$ are $G(T) \ominus \Ga _0 (T)$ and  $G(T^*) \ominus \Ga _0 (T^*)$, respectively. Here, $$ \Ga _0 (T) := (F^2 _d \oplus 0 ) \cap G(T), \quad \mbox{and} \quad  \Ga _0 (T^*) := (F^2 _d \oplus 0 ) \cap G(T^*), $$ with orthogonal projections $I-P_0$ and $I-Q_0$, respectively. 
 
All four spaces $\mc{D}, \mc{R}, \mc{D} _*, \mc{R} _*$ are contractively contained in $F^2 _d = H^2 (\B _\N ^d )$. It follows that we can view all four spaces as NC-RKHS on $\B _\N ^d$. The following theorem shows that these NC-RKHS can be identified with the operator range spaces and free right deBranges-Rovnyak spaces of the operator-valued right multipliers $A, B \in \scr{R} _d (\C ^N , \C )$ (see Subsection \ref{freedBRspace}).  

\begin{thm}
As NC-RKHS, we have that $\mc{D} = \scr{M} ^R (A), \mc{R} = \scr{M} ^R (B), \mc{D} _* = \scr{H} ^R (B)$, and $\mc{R} _* = \scr{H} ^R (A)$. The norms of any $x \in \mc{D}$ or $x_* \in \mc{D} _*$ are equal to the graph norms of $x \oplus Tx$ and $x_* \oplus T^* x_*$, respectively. The norms of any $Tx \in \mc{R}$ or $T^* x_* \in \mc{R} _*$ are the graph norms of $P_0 (x \oplus Tx)$ or $Q_0 (x_* \oplus T^* x_*)$ respectively. (If $T\sim R^\infty _d$ then $P_0 = I$.)
\end{thm}

\begin{proof}
We have that 
$$G(T) = \ran{\Theta  (R)} = \scr{M} ^R (\Theta ); \quad \quad G(T^*) = (JG(T) ) ^\perp = \scr{H} ^R (J\Theta ), $$ are both $\C ^2$-valued NC-RKHS which are isometrically contained in $F^2 _d \otimes \C ^2$ (since $\Theta$ and $J\Theta$ are both inner). The NC kernels for these spaces are then: 
$$ K^T (Z,W) = \bbm K(Z,W) [A^\dag (Z)(\cdot ) A^\dag (W)^*] & K(Z,W)[A^\dag (Z) (\cdot) B^\dag (W)^*] \\ K(Z,W)[B^\dag (Z) (\cdot) A^\dag (W)^*] & K(Z,W)[B^\dag (Z) (\cdot ) B^\dag (W) ^*] \ebm,$$ and $ K^{T^*} (Z,W) = $
$$ \bbm K(Z,W)- K(Z,W) [B^\dag (Z)(\cdot ) B^\dag (W)^*] & K(Z,W)[B^\dag (Z) (\cdot) A^\dag (W)^*] \\ K(Z,W)[A^\dag (Z) (\cdot) B^\dag(W)^*] & K(Z,W) - K(Z,W)[A^\dag (Z) (\cdot ) A^\dag (W) ^*], \ebm $$ where $K(Z,W)$, as always, is the NC Szeg\"{o} kernel of $H^2 (\B _\N ^d)$. We immediately recognize the diagonal components of these two kernels as the CPNC kernels for the range spaces and deBranges-Rovnyak spaces for $A,B$. 

Since $U ^*, V^*$ act as projection onto the first and second components of vectors in $G(T)$, and $U_* ^*, V_* ^*$ do the same for $G(T^*)$, it follows by RKHS theory that $U ^*, V^*, U_* ^* ,V_* ^*$ are co-isometric multipliers onto $\scr{M} ^R (A), \scr{M} ^R (B), \scr{H} ^R (B)$, and $\scr{H} ^R (A)$, respectively, and this proves that $\mc{D} = \scr{M} ^R (A), \mc{R} = \scr{M} ^R (B), \mc{D} _* = \scr{H} ^R (B)$, and $\mc{R} _* =\scr{H} ^R (A)$.

Moreover, since, as discussed above, $X, X_*$ are both injective, $U, U_*$ are onto isometries. This implies that the norm of any $x \in \mc{D} = \dom{T}$ or $x_* \in \mc{D} _*$ is simply equal to the graph norm of $x \oplus Tx \in G(T)$ or $x_* \oplus Tx_* \in G(T^*)$. The norm of any $Tx \in \mc{R}$ or $T^* x_* \in \mc{R} _*$ are equal to that of $P_0 (x \oplus Tx)$ and $Q_0 (x_* \oplus T^* x_*)$, where $P_0, Q_0$ project onto the ranges of $V,V_*$, which are $G(T) \ominus \Ga _0 (T)$ and $G(T^*) \ominus \Ga _0 (T^*)$, respectively.
\end{proof}
\begin{lemma}
    The quasi-affinity $X^* $ is equal to $\Theta (R) A(R) ^*$.
\end{lemma}
\begin{proof}
    Set $E:= X V$, where $V:\mc{D} \rightarrow G(T)$ is the canonical unitary. Then $E$ is the contractive embedding of $\mc{D}$ into $F^2 _d$.
For any $F \in \mc{D}$,
\ba \ip{K^\mc{D} \{ Z, y , v \} }{F} _\mc{D} & = & \ip{\kz }{F} _{F^2} \nn \\
& = & \ip{E^* \kz }{F} _{\mc{D}}, \nn \ea so that $E^* \kz = K^\mc{D} \{ Z,y,v \}$. Hence
\ba \ip{ \kz}{EE^* K \{ W,y',v'\}} & = & \ip{K^\mc{D} \{ Z,y, v\}}{K^\mc{D} \{ W,y',v' \}} _{\mc{D}} \nn \\
& = & \ip{\kz }{A(R) A(R) ^* K \{ W,y',v'\}} _{F^2}. \nn \ea This proves that $EE^* = XVV^* X^* = X X ^* = A(R) A(R) ^*$.
Also, if $X ^* F = G \oplus TG \in G(T)$ then $G = X X ^* F$, so that we obtain
\ba X^* \kz & = & \bbm  XX ^* \kz \\ T XX^* \kz \ebm \nn \\
& = & \bbm A(R) A(R) ^* \kz \\ B(R) A(R) ^* \kz \ebm  \nn \\
& = & \Theta (R) A(R) ^* \kz, \nn \ea proving the claim.
\end{proof}

\begin{remark}
As shown previously, $X^*F = \Delta _T ^{-1} F \oplus T \Delta _T ^{-1} F$ with $\Delta _T := (I+T^*T)$, and we conclude that
$\ran{XX^*} = \ran{\Delta _T ^{-1} } = \dom{\Delta _T} = \dom{T^* T}$. It follows that
$$ \dom{T^* T} = \ran{A(R) A(R) ^*} =\scr{M} (A(R) A(R) ^*)$$ and this range space is contractively contained in $\scr{M} ^R (A) = \dom{T}$ \cite[Chapter 16]{FM2}.
Alternatively,
\ba \dom{T^* T} & := & \{ F \in \dom{T} | \ TF \in \dom{T^* } \} \nn \\
& = & \{ F \in \scr{M} ^R (A) | \ TF \in \scr{M} ^R (B) \cap \scr{H} ^R (B) \}. \nn \ea
Also note that the overlapping space
$$ \ran{T} \cap \dom{T^*} = \scr{M} ^R (B) \cap \scr{H} ^R (B) = B \scr{H} ^R (B^*) = B \scr{M} ^R (A^*). $$ This follows from the general theory of operator range spaces \cite[Chapter 16]{FM2} (the last equality from the fact that $\Theta$ is inner).
\end{remark}

\subsection{Right affiliated is right Smirnov}

To prove that any $T \sim R^\infty _d$ acts as right multiplication by a right free Smirnov function, it remains to prove (by Corollary \ref{Smircor}) that the wandering space of $G(T)$ is spanned by the wandering vector $X^*1$. 

\begin{prop} \label{Linvprop}
Suppose that $T \lesssim R^\infty _d$ and $\dom{T^*}$ is $L_j$-invariant. Then the wandering space $\mc{W} (T)$ is spanned by $X^*1$.
\end{prop}

\begin{proof}
Suppose that $F \oplus TF \perp X^* 1$ is also a wandering vector for $(L \otimes I_2)$ restricted to $G(T) \otimes \C ^d$.
Then, as discussed in Remark \ref{orthoX}, $F \oplus TF = L \mbf{F} \oplus T L \mbf{F}$ (since it is orthogonal to $X^*1$), and also, we necessarily have that
$$ L \mbf{F} \oplus TL \mbf{F} \perp (L \otimes I_2) (G (T) \otimes \C ^d), $$ since $F \oplus TF$ is a wandering vector.

Hence, for any $G \in \dom{T^*T}$, and $1 \leq k \leq d$,
\begin{align*} 0 & =  \ip{ L \mbf{F} \oplus T L \mbf{F} }{L_k G \oplus T L_k G } && \\
& =  \ip{F_k}{G} + \ip{L\mbf{F}}{T ^* L_k T G } && (\dom{T^*} \ \mbox{is $L_k-$invariant})  \nn \\
& =  \ip{F_k}{G} + \ip{\mbf{F}}{L^* T^* L_k TG} && \\
& =  \ip{F_k}{G} + \ip{\mbf{F}}{(T^* \otimes I_d) L^* L_k T G } && (\dom{T^*} \ \mbox{is $L_k ^*-$invariant})  \\
& = \ip{F_k}{G} + \ip{F_k}{T^*TG} = \ip{F_k}{(I+T^*T) G}. && \end{align*}
This proves that $F_k \perp \ran{(I+T^*T)}$, but $(I+T^*T)$ is onto $\mc{D} _T = \dom{T} ^{-\| \cdot \|}$, by Theorem \ref{vNthm}, so that $F_k \equiv 0$.
Hence $X^*1$ spans the wandering subspace, and the wandering space is one-dimensional.
\end{proof}

It remains to determine, when, given $B \in \scr{R} _d (\H, \C )$, the right free deBranges-Rovnyak space $\scr{H} ^R (B)$ is $L-$invariant.  This property is closely related to the concept of a column-extreme (or quasi-extreme) Schur multiplier as introduced in the scalar, commutative setting for Drury-Arveson space in \cite{Jur2014AC} and studied in the operator-valued and free settings in \cite{JM-AC,JMqe,MMdBR,JMfree,JMfreeNC}. There are several equivalent definitions of column-extreme (CE) Schur multipliers. One can define $B \in \scr{R} _d (\H , \J)$ to be column-extreme if there is no non-zero $A \in R^\infty _d \otimes \L (\H, \J)$ so that 
$$ \bbm A \\ B \ebm \in \scr{R} _d (\H , \J \otimes \C ^2 ), $$ (see \cite{JMqe}, \cite[Section 6]{JMfreeNC}). Any column-extreme multiplier is necessarily an extreme point of the Schur class, and in the single-variable, scalar-valued setting, a Schur multiplier is extreme if and only if it is column-extreme. (It is unknown if this converse holds in general.)  The following lemmas are special cases of the results of \cite[Section 6]{JMfreeNC}:

\begin{lemma} \label{nonQElemma}
Given $B \in \scr{R} _d$, the following are equivalent:
\bn
    \item $B$ is non-CE.
    \item $B \in \scr{H} ^R (B)$.
    \item $\scr{H} ^R (B)$ is $L-$invariant.
\en
\end{lemma}

\begin{lemma}{ (\cite[Corollary 6.10]{JMfreeNC})} \label{Linvlem}
    If $B \in \scr{R} _d (\H , \C)$ is such that $Bh \in \scr{H} ^R (B)$ for every $h \in \H$, then 
$\scr{H} ^R (B)$ is $L-$invariant.
\end{lemma}

\begin{cor}
    If $T \sim R^\infty _d$ then $T \sim \scr{N} _d ^+  (R)$ if and only if $\dom{T^*}$ is $L-$invariant.
\end{cor}

\begin{proof}
    If $T \sim R^\infty _d$ and $\dom{T^*}$ is $L-$invariant, Proposition \ref{Linvprop} proves that the wandering space of $G(T)$ is one-dimensional so that $T \sim \scr{N} _d ^+ (R)$ by Corollary \ref{Smircor}. Conversely if $T \sim \scr{N}  _d ^+ (R)$ then 
    $$ G(T) = \ran{\bbm A(R) \\ B(R) \ebm} \in \scr{R}_d (\C , \C ^2 ), $$ and $A(R) \neq 0$ so that $B(R)$ is necessarily non-CE and $\dom{T^*} = \scr{H} ^R (B)$. Lemma \ref{nonQElemma} then implies the claim.
\end{proof}

\begin{prop} \label{Linvprop2}
If $T \lesssim R^\infty _d$ with $\dom{T^*} = \scr{H} ^R (B)$ and $B(R) \in \scr{R} _d (\H, \C )$, then $\dim{\H} =1$ and $\mc{W} (T)$ is spanned by  $ X^*1$.
\end{prop}
To prove this proposition we will employ several results from \cite[Section 6]{JMfreeNC}.
\begin{proof}
Recall, by definition, that $X^* 1 \neq 0$ is wandering for $G(T)$, and we choose $\Theta _1 (R) := M^R _{\theta _1}$ with $\theta _1 := X^*1 / \| X^* 1 \|$ as the first unit vector in a wandering basis for the wandering space $\mc{W} (T)$ of $G(T)$. Suppose that $\theta _2$ is a second unit wandering vector orthogonal to $\theta _1$ and define 
$$ \Theta ' (R) := \bpm \Theta _1 (R), & \Theta _2 (R) \epm =: \bpm A' (R) \\ B' (R) \epm, $$ 
so that $\ran{\Theta ' (R) } \subseteq \ran{\Theta (R)} = G(T)$. We will prove that $\theta _2 \equiv 0$, and this contradiction implies the claim. Define $T' : \dom{T'} \rightarrow F^2 _d$ by $G(T') := \ran{\Theta ' (R)}$. By construction, $T' \lesssim R^\infty _d$, and by the previous subsection, $\dom{(T') ^*} = \scr{H} ^R (B')$.

Set $\mbf{B} ' := L^* B '$, this is the contractive Gleason solution for $B'$, and it takes values in $\scr{H} ^R (B') \otimes \C ^d$ (see \cite[Section 5]{JMfreeNC}). Define the positive semi-definite operator $0 \leq b_\emptyset \in \L (\C ^2 )$ by 
$$ b_\emptyset ^2 := I - B'(0) ^* B'(0) - (\mbf{B}') ^* \mbf{B}'. $$ 
Let $\mbf{c}$ be any unit vector in $\C ^2$. Since $A' \mbf{c} \neq 0$, it follows that $B_\mbf{c} ' := B' \mbf{c} \in \scr{R} _d$ is non-CE. Also, by uniqueness of contractive Gleason solutions \cite[Theorem 5.5]{JMfreeNC}, $\mbf{B} '  _\mbf{c} = \mbf{B} ' \mbf{c} = L^* (B' \mbf{c})$, and since $\mbf{B} ' _\mbf{c}$ must be non-extremal by \cite[Theorem 6.3]{JMfreeNC}, 
\ba \left( \mbf{c}, b_\emptyset ^2 \mbf{c} \right) _{\C ^2} & = & \| \mbf{c} \| ^2 - \| B' (0) \mbf{c} \| ^2 - \| \mbf{B} ' \mbf{c} \| ^2 \nn \\ 
& = & 1 - | B' _\mbf{c} (0) ) | ^2 - \| \mbf{B} '  _\mbf{c} \| ^2 > 0. \nn \ea This proves that $b_\emptyset$ is injective and has dense range, and hence it is an invertible operator on $\C ^2$. Since $b_\emptyset$ is invertible \cite[Claim 6.7, Proposition 6.8]{JMfreeNC} imply that $B' \C ^2 \subset \scr{H} ^R (B')$ and $\scr{H} ^R (B') = \dom{(T')^*}$ is $L$-invariant by Lemma \ref{Linvlem}. In this case Proposition \ref{Linvprop} implies that $\theta _1 = X^*1$ spans the wandering space for $G(T')$ so that $\theta _2 \equiv 0$.
\end{proof}

\begin{cor} \label{main1}
Let $T : \dom{T} \subseteq F^2 _d \rightarrow F^2 _d$ be a closed linear transformation. $T$ is densely-defined and affiliated to the right free shift, $T \sim R^\infty _d$, if and only if $T$ acts as right multiplication by a holomorphic right free Smirnov function, $T \sim \scr{N} _d ^+ (R)$.
\end{cor}

\begin{cor} \label{main2}
If $T \lesssim R^\infty _d$ then 
\bn
    \item $\dom{T^*}$ is $L$-invariant.
    \item $L^\alpha \Delta _T ^{-1}1 \in \dom{T^*T}$ for every $\alpha \in \F ^d$.
    \item $G(T) = \ran{\Theta (R)}$ where $\Theta (R) 1 = X^* 1 = \Delta _T ^{-1} 1 \oplus T \Delta _T ^{-1} 1$ and $\Theta (R)$ is right-inner. Hence 
    $$ \Theta (R) = \bbm A (R) \\ B(R) \ebm \in \scr{R} _d (\C , \C ^2 ), $$ and $A,B$ have symbols $\Delta _T ^{-1} 1$ and $T \Delta _T ^{-1} 1$, respectively. If $T \sim R^\infty _d$ then $A$ is also outer.
\en
\end{cor}
\begin{proof}
    Since $\Theta = \bbm A \\ B \ebm$ is Schur, and $A \neq 0$, $B$ is non-CE and $\mc{D} _* = \dom{T^*} = \scr{H} ^R (B)$ is necessarily $L$-invariant. In particular, since $\Delta _T ^{-1} 1  \in \dom{T^* T}$, it follows that $T \Delta _T ^{-1} 1 \in \dom{T ^*}$, and hence $L^\alpha T \Delta _T ^{-1} 1 = T L^\alpha \Delta _T ^{-1} 1 \in \dom{T^*}$ by $L$-invariance of $\dom{T^*}$.
\end{proof}

\section{Factorization for the free Smirnov class}

\begin{defn}
 Given $F \in \scr{N} _d ^+ (R)$, a pair $(A , B ) \in \scr{R} _d \times \scr{R} _d$ is called an inner-outer pair for $F$ if $F(Z) = A ^\dag(Z) ^{-1} B ^\dag (Z)$, $A  \in \scr{R} _d$ is right-outer, and the column $$ \Theta  (R) := \bpm A (R) \\ B  (R) \epm \in \scr{R} _d (\C , \C ^2)$$ is right-inner. We say that $(A,B)$ is \emph{maximal} if the range of $\Theta (R)$ is the graph of $T:= M^R _{F(Z)}$, defined on its maximal domain.
\end{defn}
\begin{cor} \label{denominf2}
    Any right free Smirnov $F \in \scr{N} _d ^+ (R)$ has a unique maximal inner-outer pair $(A_F ,B_F )$. We have that $F^2 _d \subseteq \scr{N} _d ^+ (R)$, and $F \in F^2_d$ if and only if the outer $A_F $ is such that $A_F ^{-1} 1 \in F^2 _d$.
\end{cor}
\begin{lemma} \label{F2char}
Let $H$ be a free NC function on $\B _\N ^d$, and assume that $M^R _H$ is densely-defined so that $H(Z) \in \scr{N} _d ^+ (R)$. Then $H \in F^2 _d$ if and only
if $1 \in \dom{M^R _H}$. If $H \in F^2 _d$ then $L^\infty _d 1 = R^\infty _d 1 \subseteq \dom{M^R _H}$.
\end{lemma}
\begin{proof}
If $H$ is a free holomorphic NC function, and $T:= M^R _H$ is densely-defined, then it is closed (on its maximal domain) and affiliated to the right free shift by Lemma \ref{BabySmirnov} so that $H(Z) \in \scr{N} _d ^+ (R)$ by Corollary \ref{main1}. If, furthermore, $1 \in \dom{M^R _H}$ then $H = H1 \in F^2 _d$. Conversely, if $H \in F^2 _d$ then $1 \in \dom{M^R _H}$.
Since $M^R _H \sim R^\infty _d$, we have that $\bigvee L^\alpha 1 \in \dom{M^R _H}$. Given any $F \in L^\infty _d$, one can approximate $F$ by NC polynomials $p_n (L)$ so that $p_n (L) \rightarrow F(L)$ in the strong operator topology (\emph{e.g.} take Ces\`aro sums). Then, $p_n(L)1 \rightarrow F(L)1$, and
$$ M^R _H p_n (L) 1 = p_n (L ) H \rightarrow F(L) H = M_H ^R F(L) 1 $$ since $p_n (L) \rightarrow F (L)$ in SOT and $H \in F^2 _d$.
\end{proof}
\begin{proof}{ (of Corollary \ref{denominf2})}
Setting $T:= M^R _{F(Z)} \sim R^\infty _d$, the first part of the claim follows from Corollary \ref{main2}. By Lemma \ref{F2char}, if $F \in F^2 _d$ then $1 \in \dom{T}$. Since $\dom{T} = \ran{M^R _{A_F}}$, there is an $x \in F^2 _d$ such that $M^R _{A_F} x = 1$, and it follows that $x = A_F ^{-1} 1$. That is, $x(Z) = A_F (Z) ^{-1}$, for any $Z \in \B ^d _\N$.
\end{proof}
\begin{cor} \label{uniquefact}
Given $F \in \scr{N} _d ^+ (R)$, suppose that  $F^\dag (Z) = D^\dag (Z) ^{-1} N^\dag (Z)$ with $D, N \in R^\infty _d$, and $N$ right-outer. Then there is a $C \in R^\infty _d$ with $C^\dag (Z)$ invertible on the NC unit ball so that $D =A_F C$ and $N =B_F C$. In particular, if $(D,N)$ is another inner-outer pair, then $C(R)$ is inner, invertible on the NC ball, and $\ran{C(R)} ^\perp \cap \ran{A_F (R)^*} =\{ 0 \}.$
\end{cor}
This corollary follows from the free Douglas Factorization Property:
\begin{thm} \label{FDFP}
Let $A \in R^\infty _d (\H , \K)$ and $B \in R^\infty _d (\J , \K )$ be right free multipliers such that $\ran{A(R)} \subset \ran{B(R)}$. Them there is a unique right free multiplier
$C \in R^\infty _d (\H , \J )$ so that $A = B C $, $\ker{C(R)} \subset \ker{A(R) }$ and
$$ \| C \| ^2 = \inf \{ \la ^2 \geq 0 | \ A A ^* \leq \la ^2 B B ^* \}. $$
\end{thm}
This theorem can be proven using the 
Douglas Factorization Lemma \cite{DouglasLemma} and communtant lifting for row contractions \cite{Ball2001-lift,DL2010commutant} (see \cite{McTrent-DFP} for the abelian analogue). 
\begin{proof}{ (of Corollary \ref{uniquefact})}
Define $T:= M^R _F$ on its maximal domain. Lemma \ref{corelem} implies that the span of the kernel vectors $K \{ Z , y , v \}$ is a core for $T^*$. Also by assumption, if $\hat{T}$ acts as multiplication by $\hat{T} (Z) = D^\dag (Z) ^{-1} N^\dag (Z) = F(Z)$ on $\ran{D(R)}$, then 
$$ \hat{T} ^* \kz  = K \{ Z , y , F(Z) v  \} = T^* \kz,$$ and we conclude that $T^* \subseteq \hat{T} ^*$ (since the span of the $\kz$ is a core for $T^*$). Hence, $\hat{T} \subseteq T$ so that 
$$ G (\hat{T} ) = \ran{\bpm D(R) \\ N(R) \epm} ^{-\| \cdot \| } \subseteq G(T) = \ran{\bpm A_F(R) \\ B_F (R) \epm}. $$ By free Douglas factorization, there is a $C \in R^\infty _d$ so that 
$$ \bpm D(R) \\ N(R) \epm  = \bpm A_F (R) \\ B_F (R) \epm C(R) =  \bpm A_F (R) C(R) \\ B_F (R) C(R) \epm . $$ Since both $D, A_F $ are outer, $D ^\dag (Z)$ and $A_F ^\dag (Z)$ are invertible for any $Z \in \B ^d _\N$, and it follows that $C$ must have the same property. If $(D,N)$ is another inner-outer pair for $F$, then $C(R)$ must be a singular inner right multiplier, \emph{i.e.} an inner which is invertible on the NC unit ball. Moreover, since $A_F (R), D(R)$ are outer, it must be that $C(R)^* A_F (R) ^* = D (R) ^*$ is injective, so that $\ran{A_F (R) ^*} \cap \ran{C(R)} ^\perp = \{ 0 \}$.  
\end{proof}
\section{Appendix: the Free Pick interpolation theorem and the Free Smirnov class}
\label{freePick}
In this section, we show that the proof, \cite[Theorem 10.3]{Alpay}, that vectors in the Drury-Arveson space belong to the multi-variable Smirnov class, can be made to work in the free setting of $H^2 (\B ^d _\N )$. To accomplish this, we will need free function theory analogues of
the Pick and Leech theorems \cite[Theorem 2.1, Theorem 3.1]{Alpay}. The recent reference \cite{BMV2} proves free Pick and Leech theorems in a much more general setting, and \cite{Sha2017-Pick} also proves the free Pick theorem.

In this free setting the (operator-valued) Nevanlinna-Pick (NP) interpolation problem is:
Let $\H, \K$ be two Hilbert spaces. Given the initial data:
$$\mbox{\textbf{(NP)}} \quad \left\{ \begin{array}{l}  1 \leq k \leq N \\  Z^{(k)} \in \B ^d _{n_k}    \\  W_k \in \C ^{n_k \times n_k} \otimes \L (\H , \K)  \end{array} \right., $$
when can we find an element $B \in \scr{L} _d (\H, \K )$ (or in the right operator-valued Schur class)
so that $$ B(Z^{(k)}) = W_k?$$

The answer is given by:

\begin{thm}{ (free Pick theorem)}
The left Nevanlinna-Pick problem has a solution if and only if,
$$ \left[ K (Z^{(k)} , Z^{(j)} )  \otimes  I_\K - W_k \left( K (Z^{(k)} , Z^{(j)} )  \otimes I_\H \right) W_j ^* \right] _{1 \leq j,k \leq N }\geq 0.$$

The right NP problem has a solution if and only if
$$ \left[ K (Z^{(k)} , Z^{(j)} ) \otimes I_\K  -  K (Z^{(k)} , Z^{(j)} ) [W_k  \left( \cdot \otimes I_\H \right)  W_j ^* ]  \right] \geq 0.$$
\end{thm}

\begin{proof}
This is essentially the same commutant lifting proof of the Pick theorem due to D. Sarason \cite{Sarason-commlift}. Consider the left case. The necessity is as usual, if $F$ is a left free multiplier satisfying (NP), then
$I - F F ^* \geq 0$ so that
$$ \sum _{k,j =1} ^N (K_{Z ^{(k)}} \otimes I _\K ) ^*  (I - FF^* ) (K_{Z^{(j)}} \otimes I_\K ) \geq 0,$$ and it readily follows that the Pick matrix map in the theorem statement is positive semi-definite.

The converse follows by considering the (scalar) right multiplier co-invariant subspaces:
$$ M _1 ^\perp := \bigvee _{1 \leq k \leq N} K_{Z^{(k)}} \C ^{n_k \times n_k} \otimes \K, $$
$$ M_2 ^\perp := \bigvee _{1 \leq k \leq N} K_{Z^{(k)}} \C ^{n_k \times n_k} \otimes \H,$$  and defining the linear operator
$$ F ^* K_{Z^{(k)}} (A \otimes g) := K_{Z^{(k)}} W_k ^* (A \otimes g) \in M_2 ^\perp, $$ where $A \in \C ^{n_k\times n_k}$.  It is easy to check that $F^*$ intertwines the restrictions of $R^* \otimes I_\K$ and $R^* \otimes I_\H$ to $M_1 ^\perp$ and $M_2^\perp$. Since the Pick matrix corresponding to the interpolation data $W_k$ is assumed to be positive semi-definite, $F$ is a contraction which intertwines the compressions of the operator-valued right free shifts $R \otimes I_\H$  and $R \otimes I_\K$ to $M_2 ^\perp$ and $M_1 ^\perp$ respectively. By commutant lifting, there is a Schur class $\hat{F} \in \scr{L } _d$  so that $\hat{F} ^* | _{M_1 ^\perp} = F^*$, and $\hat{F}$ solves the Nevanlinna-Pick problem \cite{Ball2001-lift,DL2010commutant}.
\end{proof}

\begin{cor}{ (free Leech theorem)}
Let $A, B$ be two free functions from $\B _\N ^d$ to $\C _{nc} \otimes \L (\H _1 , \J)$ and $\C _{nc} \otimes \L (\H _2 , \J)$,
respectively, and assume that
$$ B(Z)(K(Z,W)\otimes I_{\H _2}) B(W) ^* - A(Z) (K(Z,W) \otimes I_{\H _1} ) A(W) ^*, $$ is a CPNC kernel on $\B _\N ^d$. Then there is a left free Schur multiplier $C \in \scr{L} _d (\H_1 , \H _2)$ so that $A = BC$.
Alternatively, if 
$$ K(Z,W) \otimes \mr{id}_\J \left[ B(Z) ( [\cdot ] \otimes I_{\H _2} ) B(W) ^* \right]  - K(Z,W) \otimes \mr{id} _\J  \left[ A(Z) ( [\cdot] \otimes I_{\H _1} ) A(W) ^* \right] \geq 0, $$ is a positive NC kernel then there is a right free Schur multiplier $C ^\dag (R) \in \scr{R} _d ( \H _1 , \H _2  )$ so that $A(Z) = B(Z) \bullet _R C(Z)$.
\end{cor}
In the above, recall that $\bullet _R$ denotes the `right product' of operator-valued free holomorphic functions, see Remark \ref{rightproduct}. The free Leech theorem can be viewed as a sort of extension of the Douglas Factorization Lemma \cite{DouglasLemma} (or really of the free Douglas Factorization Property, Theorem \ref{FDFP}). 
\begin{proof}
The proof is similar to the that of the abelian Leech theorem in \cite[Theorem 3.1]{Alpay}, we will prove the left version: Choose $Z^{(k)} \in \B ^d _{n_k}$ for $1 \leq k \leq N$. Define the (scalar) right multiplier co-invariant spaces:
$$ \K _2 ^N := \bigvee _{1\leq k \leq N} \ran{K_{Z^{(k)}} (B(Z^{(k)}) ^* \ \cdot )}, $$ and 
$$ \K _1 ^N:= \bigvee _{1\leq k \leq N} \ran{K_{Z^{(k)}} (A(Z^{(k)}) ^* \ \cdot )}, $$ respectively.
Since 
$$ G(Z,W) := B(Z)(K(Z,W)\otimes I_\J) B(W) ^* - A(Z) (K(Z,W) \otimes I_\J ) A(W) ^* \geq 0,$$ is assumed to be positive semi-definite, it follows that the linear map $C_N ^* : \K _2 ^N \rightarrow \K _1 ^N$ defined by 
$$   C_N^* K_Z (B(Z) ^* (yv^* \otimes g)  )= K_Z (A(Z) ^* (yv^* \otimes g)  ), $$
is a contraction. It is also clear that $C_N^*$ intertwines the restrictions of the adjoints of the operator-valued right free shifts $R_k  \otimes I_{\H _1}$ and $R_k  \otimes I_{\H _2}$, and the positivity assumption implies that $C_N$ is a contraction. As before, we apply commutant lifting to conclude that $C_N ^*$ is the restriction of the adjoint of some $\hat{C} _N \in \scr{L} _d (\H _1, \H _2)$ which obeys 
$$ A(Z ^{(k)}) = B(Z ^{(k)}) C_N (Z ^{(k)}); \quad \quad 1 \leq k \leq N. $$
Now choose a sequence $Z^{(k)} \in \B ^d _{n_k},$ for every $k \in \N$ so that this sequence is matrix-norm dense in $\B ^d _\N$. For each $N$, applying the previous analysis to the finite sequence $ ( Z^{(k)} ) _{k=1} ^N$ produces a sequence $\hat{C} _N$ of left multipliers which are uniformly norm bounded (by one since they are all Schur). By WOT-compactness, there is a subsequence $C_{n} := \hat{C} _{N_n}$ which converge in the weak operator topology to some $C \in \scr{L} _d (\H _1 , \H _2 )$. It follows that for a matrix-norm dense set of $Z \in \B ^d _\N$,
$$ A(Z) = B(Z)C(Z), $$ so that by continuity this holds for all $Z \in \B ^d _\N$, and $A = BC$. Proof of the right Leech theorem is analogous.
\end{proof}

As in \cite{Alpay}, or \cite[Theorem 1.1]{Smirnov}, this can be applied to give a Smirnov characterization of $F^2 _d$:

\begin{thm} \label{freeAlpay}
Suppose that $F \in H^2 (\B ^d _\N )$ and $\| F \| _{H^2} \leq 1$. Then there are left Schur $A(L), B(L) \in \scr{L} _d$, and right Schur $C(R), D(R) \in \scr{R} _d$ so that $B _\emptyset, D_\emptyset =0$ and
$$ F(Z) = A(Z) (I - B(Z) ) ^{-1} = (I - D^\dag (Z) ) ^{-1} C^\dag (Z). $$
\end{thm}

Note that the above does not necessarily imply that $(I-B) ^{-1} \in F^2 _d$, although $(I-B(Z))$ is necessarily invertible on the NC unit ball, by Lemma \ref{freeouter} and Lemma \ref{freeouter1}. Recall that Corollary \ref{denominf2} provides an alternate Smirnov factorization of any $F \in F^2 _d$ as $F(Z) = B(Z) A(Z) ^{-1}$ where $A,B \in \scr{L} _d$, $A$ is outer, and $1/A \in F^2 _d$ (and there is also a corresponding right factorization).

The following is a general fact from NC-RKHS theory, it is the non-commutative analogue of a classical RKHS result, \cite[Theorem 10.17]{Paulsen-rkhs}. 

\begin{lemma} \label{NCRKHSfact}
    Let $\H _{nc} (K)$ be a NC-RKHS on a NC set $\Om$. A free NC function $f$ on $\Om$ belongs to
$\H _{nc} (K)$ if and only if $\la ^2 K(Z,W) - f(Z) (\cdot ) f(W) ^* \geq 0$ is a CPNC kernel on $\Om$ for some $\la ^2 >0$. The norm
of $f$ is the infimum of all such $\la$.
\end{lemma}

\begin{proof}{ (of Theorem \ref{freeAlpay})}
    The proof is formally analogous to that of \cite[Theorem 1.1]{Smirnov}: Suppose that $f \in H^2 (\B _\N ^d )$, assume without loss of generality that $\| f \| _{H^2 (\B _\N ^d )} \leq 1$, and for $Z \in \B ^d _n$ set
$$ \Phi (Z) := \bbm I_n, & f(Z) Z \ebm : H^2 (\B _\N ^d ) \oplus H^2 (\B _\N ^d ) \otimes \C ^d \rightarrow H^2 (\B _\N ^d ). $$
Then,
\begin{align*}  &  \Phi (Z) \bbm K(Z,W) & 0 \\ 0 & K(Z,W) \otimes I_d \ebm \Phi (W) ^* \\
 & =   K(Z,W) + f(Z) Z (K(Z,W) \otimes I_d ) W^* f(W) ^*  \\
& =  K(Z,W) + f(Z) K(Z,W) f(W) ^* - f(Z) f(W) ^*.  \end{align*} In the above we used that
$$ K(Z,W)[P] = \sum _{\alpha \in \F ^d} Z^\alpha P (W^*) ^{\alpha ^\dag }$$ to conclude that
\ba Z K(Z,W) [P] \otimes I_d W^* & = & \sum _{\substack{1\leq j \leq d \\ \alpha \in \F ^d}} Z^{j\alpha} P (W^*) ^{\alpha ^\dag j} \nn \\
& = & \sum _{\alpha \neq \emptyset } Z^\alpha P (W^*) ^{\alpha ^\dag} = K(Z,W)[P] - P. \nn \ea This is a familiar property shared with the abelian Szeg\"{o} kernel for Drury-Arveson space.
Hence,
 \begin{align*} &  \Phi (Z) \bbm K(Z,W) & 0 \\ 0 & K(Z,W) \otimes I_d \ebm \Phi (W) ^* \\
 &- [ f(Z),  \  \underbrace{0, \cdots , 0}_{\mbox{d times}} ] \bbm  K(Z,W) & 0 \\ 0 & K(Z,W ) \otimes I_d \ebm  \bbm f(W) ^* \\ 0 \\ \vdots \\ 0 \ebm   \\
 & =  K(Z,W) -f(Z) f(W) ^* \geq 0,  \end{align*} by Lemma \ref{NCRKHSfact} (since $\| f \| _{F^2 _d} \leq 1$).
By the (left) free Leech theorem, there is a left free Schur class $\Psi \in \scr{L} _d (\C \oplus \C^d )$ so that $\Phi \Psi = (f(Z) , 0, ... , 0)$:
$$ (f(Z), \mbf{0} _d) = \Phi (Z) \Psi (Z) =: \left[ I, \ f(Z) Z \right] \bbm \Psi _{11} (Z) & \Psi _{12} (Z) \\ \Psi _{21} (Z) & \Psi _{22} (Z) \ebm. $$
Hence, $$ f(Z) = \Psi _{11} (Z) + f(Z) Z \Psi _{21} (Z). $$
Solving for $f(Z)$ yields:
$$ f(Z) ( I - Z \Psi _{21} (Z) ) = \Psi _{11} (Z). $$ Setting $A = \Psi _{11}$, and $B = Z \Psi _{21}$ we have that $A,B \in \scr{L} _d$ are Schur with $B(0) = B_\emptyset =0$, so that $I - B$ is outer by Lemma \ref{freeouter}, and
$$ f(Z) =  A(Z) (I - B(Z) ) ^{-1}, $$ belongs to the left free Smirnov class.
The proof of the corresponding `right' statement is similar, but the algebra is slightly different:
For $Z \in \B ^d _n$, set 
$$ \Phi (Z) := \bbm I_n, & Z (f(Z) \otimes I_d) \ebm, $$ so that 
$$ M ^R _\Phi : H^2 (\B ^d _\N ) \oplus H^2 (\B ^d _\N ) \otimes \C ^d \rightarrow H^2 (\B ^d _\N ). $$ 
For $Z \in \B ^d _n$, $W \in \B ^d _m$ and $P \in \C ^{n\times m}$, consider the CPNC kernel: 
\ba & &  K(Z,W) \left[ \Phi (Z) \left( P \otimes I_{d+1} \right) \Phi (W) ^* \right]  = K(Z,W) +\nn \\
&  &   K(Z,W) \left[ \bbm Z_1, & \cdots & Z_d \ebm \bbm f(Z) P f(W) ^* & & \\ & \ddots & \\ & & f(Z) P f(W) ^* \ebm \bbm W_1 ^* \\ \vdots \\ W_d ^* \ebm \right] \nn \\
& = & K(Z,W) [P] + K(Z,W) \left[ \sum _{j=1} ^d Z_j f(Z) P f(W) ^* W_j ^* \right] \nn \\
& = & K(Z,W) [P] + K(Z,W)[ f(Z) P f(W) ^*] - f(Z) P f(W)^*; \quad \quad \mbox{(as before).} \nn \ea
As before, it follows that 
$$ K(Z,W) \left[ \Phi (Z) \left( \cdot \otimes I _{d+1}  \right) \Phi (W) ^* \right] - K(Z,W) [ f(Z) (\cdot ) f(W) ^* ] \geq 0, $$ is a CPNC kernel, so that the right free Leech theorem implies that there is a $$ \Psi (Z) := \bbm \psi  (Z) \\ \psi _1 (Z) \\  \ddots \\ \psi _d (Z) \ebm, $$ such that 
$$ f(Z) =   \Phi (Z) \bullet _R \Psi (Z), $$ and $\Psi  \in \scr{R} _d (\C , \C ^d )$. Equivalently,
$$ f(Z) = \psi (Z) + \psi _1 (Z) Z_1 f(Z) + \cdots + \psi _d (Z) Z_d f(Z). $$ Solving for $f(Z)$ yields:
$$ \left( I - \underbrace{(\psi _1 (Z) Z_1 + \cdots \psi _d (Z) Z_d )}_{=: D^\dag (Z)} \right) f(Z) = \psi (Z) =: C ^\dag (Z). $$ Clearly $D, C \in \scr{R} _d$ are such that $D_\emptyset = 0$, and 
$$ f(Z) = (I - D^\dag (Z) ) ^{-1} C^\dag (Z). $$ 
\end{proof}

\subsection{Null sets of multipliers}

For any $Z \in \B ^d _n$ and $y \in \C ^n$ define the right multiplier invariant spaces:
$$ \ran{K_Z}, \quad \quad \mbox{and} \quad  \ran{K  \{ Z , y , \cdot \} }. $$ (The left versions of all results in this subsection follow by interchanging `left' and `right'.) Also consider the spaces:
$$ E_Z := \bigcap \{ \ker{\phi (L) ^*} | \ \phi (L) \in L^\infty _d \ \mbox{and} \ \phi (Z) \equiv 0\}, $$ and
$$ E_{Z,y} := \bigcap \{ \ker{\phi (L) ^*} | \ \phi (Z) ^* y =0 \}. $$ 
Clearly, $E_Z \subseteq E_{Z,y}$. Also clearly, $\ran{K_Z} \subseteq E_Z$ and $\ran{K \{ Z , y , \cdot \}} \subseteq E_{Z,y}$.

\begin{cor} \label{kernull}
The spaces $E_Z$ and $E_{Z,y}$ are equal to $\ran{K_Z}, \ran{K \{ Z, y , \cdot \} }$, respectively.
\end{cor}

\begin{proof}
Suppose that $F \in E_Z \ominus \ran{K _Z}$. By the previous Smirnov characterization of $F^2 _d$, Theorem \ref{freeAlpay}, we can write $F = F(L) 1$ where
$$ F(L) = N(L) (I - D(L) ) ^{-1}, $$ $N \in L^\infty _d$, $D \in \scr{L} _d$ and $D_\emptyset = 0$. In particular, $1-D$ is outer by Lemma \ref{freeouter}. Since we assume that $F \perp K_Z (A)$ for any $A \in \C ^{n \times n
}$, we have that for any $y,v \in \C ^n$ that,
$$ 0 = \ip{K_Z (yv^*)}{F} = \ipcn{y}{F(Z) v}, $$ and it follows that $F(Z) \equiv 0$. Since $I_n - D(Z)$ is invertible, by Lemma \ref{freeouter1}, it follows that $N (Z) \equiv 0$. By definition since $F \in E_Z$, it follows that $N(L) ^* F =0$ so that
$$ \ip{F}{F}= \ip{N(L) ^*F}{(I-D(L)) ^{-1} 1} _{F^2 } = 0. $$

Similarly, suppose that $F \in E_{Z,y} \ominus \ran{K \{ Z , y, \cdot \}}$. Then, as before, writing $F=F(L) 1 = A(L) (I-B(L) ) ^{-1} 1$, we obtain that for any $v^* \in \C ^n$,
$$ 0  = \ip{\kz}{F}  = \ip{K \{ Z, A(Z) ^* y, v \} }{(I - B(L)) ^{-1} 1 }. $$ In particular, for any right multiplier $G(R) \in R^\infty _d$,
\ba 0 & = & \ip{K \{ Z,y, G(Z)v \} }{F} \nn \\
& = & \ip{K \{ Z,A(Z) ^*y, v \}}{ (I -B(L) ) ^{-1} G(R) 1}. \nn \ea In particular, we can choose $G(R) 1 = (I -B(L) ) p(L) 1$ for any non-commutative polynomial $p$,
so that $K \{ Z,A(Z) ^* y, v\} \equiv 0$. That is,
$$ 0 = \| K \{ Z,A(Z)^* y, v \} \| ^2 = \ipcn{A(Z) ^* y}{K(Z,Z)(vv^*)A(Z) ^* y},$$ for all $v  \in C^n$, which implies that
$$ 0 = \ipcn{A(Z) ^* y}{K(Z,Z) (I_n) A(Z) ^* y}.$$ Since $K(Z,Z)(I_n)$ is invertible, $A(Z) ^* y  =0$. Hence, by assumption $A(L) ^* F =0$. But then, as before
$$ \| F \| ^2 = \ip{A(L) ^* F}{(I-B(L ) ) ^{-1} 1 } = 0.$$
\end{proof}

\begin{remark}
By \cite[Theorem 7.2, Theorem 7.4]{KVV}, $T(Z)$ will be holomorphic if and only if it is locally bounded.
The fact that any $T \sim R^\infty _d$ acts as right multiplication by a free holomorphic function in the right free Smirnov class follows from Corollary \ref{main1}.

Finally, let us remark that the results of this subsection can also be used to prove that any $T \sim R^\infty _d$ acts as right multiplication by some free function $T(Z)$, if we make the additional key assumption that the kernel vectors are contained in the domain of $T^*$:
\end{remark}

\begin{cor} \label{Kercor}
    If $T \sim R^\infty _d$ and $\bigvee _{Z \in \B ^d _\N } \ran{K_Z} \subseteq \dom{T^*}$ then $T= M^R _{T(Z)}$ acts as right multiplication by a free function $T(Z)$.
\end{cor}
\begin{proof}
By Corollary \ref{kernull} and by assumption, $ E_{Z,y} = \bigvee _{v \in \C ^n} \kz \subset \dom{T^*}$ for any $Z \in B^d _n$ and $y \in \C ^n$. Also, since $T \sim R^\infty _d$, it is clear that 
$E_{Z,y}$ is invariant for $T^*$, so that
$$ T ^* \kz = K \{ Z, y , v' \}. $$ We can then define $T(Z) \in \C ^{n \times n}$ by the formula:
$$ T(Z)v  := v'. $$ It remains to show that $T(Z)$ is a free function. First, it is clearly graded since $T(Z) \in \C ^{n\times n}$ whenever $Z \in \B ^d _n$.  To show that $T$ acts as right multiplication by a free function, it remains to show that $T(Z)$ respects intertwinings. If $Z \in \B ^d _n, \ W \in \B ^d _m$, and $\alpha \in \C ^{m\times n}$ obey:
$$ \alpha Z = W \alpha, $$ then for any $Y \in \B ^d _\N$, we have that
$$ \alpha K(Z,Y) (P ) = K(W,Y) (\alpha P), $$ this a property of NC kernels \cite[Section 2.3]{BMV}. It follows from this that
$$ K\{ W, x , \alpha v  \} = K \{ Z,\alpha ^* x, v \}, $$ for any $v \in \C ^n, \ x \in \C ^m$. Hence,
\ba K \{ W,x, \alpha T(Z) v \} & = & K \{ Z, \alpha ^* x, T(Z)v \}  \nn \\
& = & T^* K \{ Z,\alpha ^* x, v \} \nn \\
& = & T^* K \{ W,x, \alpha v \} \nn \\
& = & K \{ W,x,T(W)\alpha v \} \nn, \ea and this proves that
$$ \alpha T(Z) = T(W) \alpha, $$ so that $T(Z)$ respects intertwinings.
\end{proof}

\bibliography{Smirnov}

\end{document}